\documentclass[11pt]{amsart}
\usepackage{mathtools}
\usepackage{amsthm}
\usepackage{amsmath}
\usepackage{amssymb}
\usepackage{mathrsfs}
\usepackage[style=trad-alpha]{biblatex}
\calclayout

\usepackage[shortlabels]{enumitem} 
\usepackage[mathscr]{euscript}
\usepackage{multicol}

\usepackage{bbm} 

\usepackage{tikz-cd} 

\usepackage{centernot}

\usepackage{color} 

\usepackage{todonotes}

\usepackage[T1]{fontenc} 

 \newtheorem{theorem}{Theorem}[section] 

\theoremstyle{definition}
\newtheorem{definition}[theorem]{Definition} 
\newtheorem{proposition}[theorem]{Proposition}
\newtheorem{corollary}[theorem]{Corollary}
\newtheorem{lemma}[theorem]{Lemma}

\newtheorem{example}[theorem]{Example} 

\newtheorem{question}[theorem]{Question}

\newtheorem{notation}[theorem]{Notation}

\numberwithin{equation}{subsection}

\theoremstyle{remark}
\newtheorem{remark}[theorem]{Remark} 

\numberwithin{equation}{section}

\DeclareMathOperator{\id}{id}

\DeclareMathOperator{\im}{im}
\DeclareMathOperator{\Spec}{Spec}

\DeclareMathOperator{\rk}{rk}
\DeclareMathOperator{\Hom}{Hom}

\DeclareMathOperator{\uAut}{\underline{Aut}}
\DeclareMathOperator{\stab}{Stab}

\DeclareMathOperator{\ind}{Ind}
\DeclareMathOperator{\Char}{char}

\DeclareMathOperator{\Aut}{Aut}

\DeclareMathOperator{\gal}{Gal} 
\DeclareMathOperator{\res}{Res} 
\DeclareMathOperator{\diag}{diag} 

\DeclareMathOperator{\ef}{\text{ef}}
\newcommand{\Ss}{\mathrm{ss}}

\DeclareMathOperator{\bun}{\mathrm{Bun}}
\DeclareMathOperator{\Bun}{\mathrm{Bun}}

\DeclareMathOperator{\Ad}{Ad}
\DeclareMathOperator{\Lie}{Lie} 

\newcommand{\fkg}{\mathfrak{g}}

\newcommand{\cale}{\mathcal{E}}
\newcommand{\calf}{\mathcal{F}}
\newcommand{\calg}{\mathcal{G}}
\newcommand{\calh}{\mathcal{H}}
\newcommand{\cali}{\mathcal{I}}

\newcommand{\call}{\mathcal{L}}

\newcommand{\calo}{\mathcal{O}}

\newcommand{\calp}{\mathcal{P}}

\newcommand{\calt}{\mathcal{T}}
\newcommand{\calx}{\mathcal{X}}
\newcommand{\caly}{\mathcal{Y}}

\newcommand{\cG}{\mathcal{G}}

\newcommand{\cX}{\mathcal{X}}
\newcommand{\cY}{\mathcal{Y}}

\newcommand{\mc}{\mathbb{C}}

\newcommand{\mg}{\mathbb{G}}

\newcommand{\rR}{\mathrm{R}}
\newcommand{\rB}{\mathrm{B}} 
\newcommand{\rH}{\mathrm{H}} 
\newcommand{\GL}{\mathrm{GL}} 
\newcommand{\GO}{\mathrm{GO}} 
\newcommand{\rO}{\mathrm{O}} 
\newcommand{\rZ}{\mathrm{Z}} 
\newcommand{\bOne}{\mathbbm{1}}
\newcommand{\Go}{G^\circ} 
\newcommand{\et}{\mathrm{\acute{e}t}} 

\newcommand{\pr}{\mathrm{pr}}

\newcommand{\tuomas}[1]{{\color{orange} (Tuomas: #1)}}
\newcommand{\stefan}[1]{{\color{blue} (Stefan: #1)}}
\newcommand{\ludvig}[1]{{\color{green} (Ludvig: #1)}}

\usepackage{hyperref}
\hypersetup{
    colorlinks=true,
    linkcolor=blue,
    filecolor=blue,      
    urlcolor=blue,
}
\usepackage{cleveref}

\title[Moduli of $\calg$-bundles and nondensity of EF bundles]{Moduli of $\calg$-bundles under nonconnected group schemes and nondensity of essentially finite bundles}
\author[L. Olsson, S. Reppen, T. Tajakka]{Ludvig Olsson, Stefan Reppen, Tuomas Tajakka}
\address[S. Reppen]{\parbox{\linewidth}{Department of Mathematics, Stockholm University \\
Graduate School of Mathematical Sciences, the University of Tokyo}}
\email[]{stefan.reppen@gmail.com}
\date{}

\begin{document}
\maketitle
\begin{abstract}
We prove the existence of a projective good moduli space of principal $\calg$-bundles under nonconnected reductive group schemes $\calg$ over a smooth projective curve $C$. We also prove that the moduli stack of $\calg$-bundles decomposes into finitely many substacks $\bun_{\calp}$ each admitting a torsor $\bun_{\calg_\calp}\to \bun_\calp$ under a finite group, for some connected reductive group schemes $\calg_\calp$ over $C$. We use this for the second purpose of the article: for any constant connected reductive group $G$, the subset of essentially finite $G$-bundles in the moduli of degree 0 semistable $G$-bundles over $C$ is not dense, unless $G$ is a torus or the genus of $C$ is smaller than 2. We do this by giving an upper bound on the dimension of the closure of the subset of essentially finite $G$-bundles. 
\end{abstract}
%
%
\section{Introduction}

Let $C$ be a smooth projective curve over an algebraically closed field $k$ and let $\calg$ be a reductive group scheme over $C$. In most studies on $\calg$-bundles over $C$ the group is usually assumed to be connected. There are both group-theoretic and geometric complications that occur in the nonconnected case. For example, Zariski-local triviality (see \cite{drinfeld.simpson}) of $\calg$–bundles does not hold for nonconnected groups. 
Nonetheless, nonconnected groups occur naturally both in their own right, such as the orthogonal group corresponding to a vector bundle with a quadratic form, and in the study of connected groups, for instance as normalizers of tori. Thus, one of the purposes of this text is to study the moduli stack $\bun_\calg$ of principal bundles under nonconnected group schemes over $C$. In particular, we show that $\bun_\calg$ decomposes into finitely many substacks $\bun_\calp$, each of which admits a finite torsor $\bun_{\calg_\calp}\to \bun_\calp$, where $\calg_\calp$ is a connected reductive group over $C$. When $\Char(k)=0$, we use this to prove that the substack of semistable $\calg$-bundles admits a projective good moduli space (see \Cref{intro.moduli} for some background on this topic). 

We also use this decomposition for the second purpose of the article, which is to settle a density question (as discussed in \cite{ducrohet.mehta},\cite{ghiasabadi.reppen},\cite{weissmann}) regarding essentially finite bundles under connected reductive groups over $k$, in the characteristic 0 case (see \Cref{density.question.main.theorem.intro}). 
%
%
\subsection{Moduli of \texorpdfstring{$\calg$}{G}-bundles}\label{intro.moduli}
The construction of moduli of $G$-bundles over $C$ has its roots in the work of Ramanathan who
constructed a projective good moduli space of semistable $G$-bundles for any connected reductive group $G$ over $\mc$ (see \cite{raman1,raman2}). After decades of work addressing the existence of such moduli spaces (see e.g. \cite{faltings}, \cite{balaji2}, \cite{balaji.semi.princ}) the theory culminated 
in the seminal work \cite{gomez.langer.schmitt.sols} by G\'omez-Langer-Sols-Schmitt, where moduli of $G$-bundles were constructed in arbitrary characteristic, for arbitrary reductive groups over $k$, over an arbitrary smooth projective base scheme.\footnote{We remark that to obtain projectivity in some of the previous works, some version of the semistable reduction theorem was needed, and this was proved by Heinloth in \cite{heinloth.semistable.reduction} for large classes of reductive groups.}

In recent years, there has been a growing interest in constructions of moduli of $\calg$-bundles for more general groups than reductive group schemes over $k$, in particular for so-called Bruhat-Tits group schemes over $C$. These are smooth affine group schemes $\calg$ over $C$ with geometrically connected fibers such that there is an open subset $U\subset C$ such that $\calg|_U$ is connected reductive and for all $x\in C\setminus U$, the group scheme $\calg|_{\Spec \calo_{C,x}}$ is a connected parahoric Bruhat-Tits group. 
In particular, $\calg$ could be a connected reductive group scheme over $C$.
When $\Char(k)=0$, Balaji and Seshadri proved in \cite{balaji.seshadri.bruhat.tits} the existence of a projective good moduli space of $\calg$-bundles in the case where $\calg|_U$ is semisimple and simply connected, and for general Bruhat-Tits group schemes Alper--Halpern-Leistner--Heinloth proved in \cite{AHLH} the existence of a proper 
good moduli space for arbitrary Bruhat-Tits group schemes.

In this paper we study moduli of principal $\calg$-bundles in the case where the group scheme $\calg$ is reductive over all of $C$, but not necessarily connected. 
We declare a $\calg$-bundle $\cale$ to be semistable if the adjoint bundle $\Ad_*\cale$ is semistable 
(see \Cref{section:semistability}), and we prove the following.
\begin{theorem}\label{intro.main.theorem.gms}
Suppose that $\Char(k)=0$, and 
let $\calg$ be a (not necessarily connected) reductive group scheme over $C$. The stack $\bun_{\calg}^{\Ss}$ of semistable $\calg$-bundles admits a projective good moduli space.
\end{theorem}
We prove this result by reducing to the case of connected groups via the following key statement.
\begin{theorem}\label{decomposition.theorem.intro}
Suppose that $\calg$ is a (not necessarily connected) reductive group over $C$. If $\Char(k)>0$ assume furthermore that that $\calg/\calg^0$ is finite and that $\Char(k)$ does not divide $|\calg/\calg^\circ|$. Then there exists a finite decomposition $\bun_{\calg} \cong \coprod_{\calp}\bun_{\calp}$ by open and closed substacks such that each $\bun_{\calp}$ admits a torsor $\bun_{\calg_\calp} \to \bun_{\calp}$ under a finite group, where $\calg_{\calp}$ is a connected reductive group scheme over $C$.
\end{theorem}
We believe that this decomposition could be useful in extending various results about $\bun_{\calg}$ from the case of connected groups to the case of disconnected groups, as illustrated in its use to prove \Cref{intro.main.theorem.gms}. It also seems useful in the study of moduli of principal bundles under connected groups. This is manifested in its crucial use in our proof of nondensity of essentially finite $G$-bundles for connected reductive groups $G$ over $k$ (see below). 

To obtain projectivity of the moduli space we use the following auxiliary result which can be deduced from \cite[Theorem 4.4.3]{behrend.thesis} and might be worth stating explicitly since it has occurred in the literature in the case of moduli of semistable vector bundles (see \cite[Theorem 4.2]{weissmann} cf. \cite[Theorem 4.2]{hein}, \cite[Theorem I.4]{faltings}).
\begin{proposition}
Let $\calg$ be an affine smooth group scheme of finite type over $C$ and suppose that there is a notion of semistability on $\bun_{\calg}$ for which the substack of semistable bundles admits a proper good moduli space $M_{\calg,C}^{\Ss}$. For any finite cover $f\colon D\to C$ such that a $\calg$–bundle $\cale$ is semistable if and only if $f^*\cale$ is semistable, $f$ induces via pullback a finite morphism $f^*\colon M_{\calg,C}^{\Ss}\to M_{f^*\calg,D}^{\Ss}$.
\end{proposition}
\subsection{Essentially finite bundles}\label{intro.ef}
Building on work by Weil in the 1930's, Nori introduced in \cite{Nori1} the notion of an essentially finite (EF) vector bundle.\footnote{He studied these bundles over more general bases and over fields which are not necessarily algebraically closed.} A vector bundle is EF if it admits a reduction to a finite group. Essentially finite vector bundles form a Tannakian category and the corresponding fundamental group $\pi_1^N(C)$ 
classifies covers of $C$ by finite group schemes (omitting base points from the notation). For example, if $\Char(k)=0$ then $\pi_1^N(C)(k)=\pi_{1}^{\et}(C)$. EF bundles have since then been studied extensively from various perspectives (e.g. \cite{esnault.et.al}, \cite{biswas.dossantos},  \cite{borne.vistoli}). If $G$ is any algebraic group one may extend the definition to $G$-bundles by saying that a $G$-bundle is EF if it admits a reduction to a finite group. A heuristic question one may ask is the following.
\begin{question}\label{how.many.ef.bundles}
How many essentially finite bundles are there?
\end{question}
Nori proved in \cite{Nori2} that EF vector bundles are semistable of degree 0 and in \cite{ghiasabadi.reppen} the second author and Ghiasabadi proved more generally that EF $G$-bundles are semistable of torsion degree, for any connected reductive group $G$ over $k$. Hence, for any such group $G$, in the moduli space $M_{G}^{\Ss,0}$ of semistable $G$-bundles of degree $0$, we have a subset $M_{G}^{\ef}\subset M_{G}^{\Ss,0}$ consisting of EF $G$-bundles of degree 0. A first approximation of \Cref{how.many.ef.bundles} is to ask about the size of $M_{G}^{\ef}$. For instance, if $G=\mg_m$, then EF bundles correspond to torsion points on the Jacobian. Since torsion points are dense in any abelian variety, we see that $M_{\mg_m}^{\ef}\subset M_{\mg_m}^{\Ss,0}$ is dense. In \cite{ducrohet.mehta} Ducrohet and Mehta showed that if $\Char(k)=p>0$, then $M_{\GL(n)}^{\ef}$ is dense for all $n\geq 1$. However, Weissmann proved in \cite{weissmann} that those EF bundles which are trivialized by a prime to $p$ cover are not dense by giving an explicit upper bound on the dimension of their Zariski closure in $ M_{\GL(n)}^{\Ss,0}$. In a similar direction, it is proved in \cite{ghiasabadi.reppen} that if $\Char(k)=0$, then $M_{G}^{\ef}$ is not dense for any connected reductive group $G$ of semisimple rank 1. In this article we settle this question of density completely in the characteristic 0 case. 
\begin{theorem}\label{density.question.main.theorem.intro}
Let $\Char(k)=0$, let $g$ denote the genus of $C$ and assume that $g\geq 2$. For any connected, reductive group $G$ we have that
\begin{equation}
    \dim\overline{M_{G}^{\textnormal{ef}}}\leq g \rk(G).
\end{equation}
In particular, if $G$ is not a torus, then $M_{G}^{\textnormal{ef}}$ is not dense in $M_{G}^{\Ss,0}$.
\end{theorem}
\begin{remark}
If $g = 0$, then $M_G^{\Ss}(k) = \{ \ast \}$, and if $g=1$, then \cite[Proposition 4.9]{ghiasabadi.reppen} shows that $M_{G}^{\ef}$ is dense for all connected reductive groups $G$. The same holds for all $g$ if 
$G \cong \mg_m^n$ is a torus, since then $M_G^{\Ss} \cong \prod_{i=1}^n M_{\mg_m}^{\Ss}$.
\end{remark}
A key input in proving \Cref{density.question.main.theorem.intro} is the following purely group theoretic result inspired by Jordan's classical result in \cite{jordan} on finite subgroups of reductive groups, which could be of independent interest.
\begin{theorem}\label{theorem.finite.subgroups}
Let $\Char(k)=0$, and let $G$ be a connected, reductive group over $k$. There exist subgroups $H_1, \ldots, H_n \subset G$ such that the connected component of identity of each $H_i$ 
is a torus and such that all finite subgroups of $G$, up to conjugation, are contained in some $H_i$.
\end{theorem}
\subsection{Outline}
In \Cref{section:notation} we introduce notation and assumptions used throughout the text. \Cref{section:moduli} concerns the moduli stack of $\calg$-bundles: in \Cref{section:decomposition} we prove that $\bun_{\calg}$ admits a decomposition as stated in \Cref{decomposition.theorem.intro} and after some discussion regarding semistability in \Cref{section:semistability} we prove in \Cref{section:existence.of.moduli} that the semistable locus $\bun_{\calg}^{\Ss}$ admits a projective good moduli space. \Cref{section:essentially.finite.bundles} is divided into two sections: in \Cref{section:finite.subgroups} we study finite subgroups of reductive groups over $k$ and we prove \Cref{theorem.finite.subgroups}. In \Cref{section:nondensity} we use the results from the previous sections to prove our main theorem on essentially finite bundles. Finally, \Cref{appendix} is devoted to a proof of a generalization of Jordan's theorem on finite subgroups of $\GL(n)$ to arbitrary reductive groups.
\subsection{Acknowledgements}
We would like to thank Magnus Carlson, Carlo Gasbarri, Siddharth Mathur, David Rydh, and Dario Weissmann for enlightening conversations. 
Part of this work was carried out while the second author was a JSPS International Research Fellow.
%
%
%
%
\section{Notation}\label{section:notation}
Let $k$ be an algebraically closed field and let $C$ be a smooth projective connected curve over $k$. From \Cref{gamma.descent.gms} onward we assume that $\Char(k)=0$.

For an affine group $G$ over a field $K$, let $G^\circ$ denote the connected component of identity. A reductive group over $K$ is a smooth affine group $G$ such that $G_{\overline{K}}^\circ$ has trivial unipotent radical. A reductive (resp. connected reductive) group over $C$ is a smooth, affine group scheme $\calg$ over $C$ whose geometric fibers $\calg_{x}$ are reductive (resp. connected reductive) groups for all geometric points $x$ in $C$. By a constant group scheme over $C$ we mean a group scheme isomorphic to $G\times_k C$ for some $k$-group $G$. For a finite type group scheme $\calg$ over $C$ we denote by $\Lie(\calg)$ the Lie algebra of $\calg$ (a vector bundle over $C$) and by $\Ad\colon \calg \to \GL(\Lie(\calg))$ the adjoint representation. If $G$ is an affine $k$-group we also use the notation $\mathfrak{g}$ for its Lie algebra.

By \cite[Proposition 3.1.3]{conrad.reductive.group.schemes}, for any reductive group scheme $\calg$ over $C$, there is an open and closed subgroup scheme $\calg^\circ$ which is a connected reductive group over $C$, and $\calg/\calg^\circ$ is a separated, \'etale, quasi-finite $C$-group. 

A $\calg$-bundle over $C$ is an fppf scheme $\cale\to C$ with a $\calg$-action such that $\cale \times_C \calg \to \cale \times_C \cale$, $(p,g)\mapsto (p,pg)$ is an isomorphism. Since all groups we consider are smooth, we note that every $\calg$-bundle is locally trivial in the \'etale topology. 
As is common in the literature we also refer to a $\calg$-bundle as a principal $\calg$-bundle or a $\calg$-torsor. Further, since $\calg$ is affine over $C$ we may view a bundle $\cale$ as a scheme over $C$ or as a sheaf. 
Conveniently but also quite uncommonly, we denote by $\bOne_{\calg}$ the trivial $\calg$-torsor.
%

We denote by $\bun_{\cG,C}$ the stack on the category of $k$-schemes whose fiber over $T$ is the groupoid of $T \times_k \cG$-torsors on $T \times_k C$. If the meaning is clear from context, we usually write $\bun_\calg$ in place of $\bun_{\calg,C}$. For a $\calg$–bundle $\cale$ over $C$, we write $\underline{\Aut}(\cale)$ for the $C$-group scheme of automorphisms of $\cale$ and we write $\Aut(\cale)$ for the $k$-group which is the stabilizer of $\cale \in \bun_{\calg}(k)$.

If $\cX$ is a stack on the category of $C$-schemes, we denote by $\res_{C/k}\cX$ the restriction of scalars of $\cX$ to $k$, defined as the stack whose fiber over a $k$-scheme $T$ is $\cX(\pr_2: T \times_k C \to C)$. 
If the left diagram below is a commutative (resp. cartesian) diagram of stacks over $C$, then the right diagram is a commutative (resp. cartesian) diagram of stacks over $k$.
\begin{center}
\begin{tikzcd}
  \cX' \arrow[r] \arrow[d] & \cX \arrow[d] \\
  \cY' \arrow[r] & \cY
\end{tikzcd}
\qquad
\begin{tikzcd}
  \res_{C/k}\cX' \arrow[r] \arrow[d] & \res_{C/k}\cX \arrow[d] \\
  \res_{C/k}\cY' \arrow[r] & \res_{C/k}\cY
\end{tikzcd}
\end{center}
We further denote by $\rB \calg$ the classifying stack of $\calg$. It follows from the definitions that $\bun_{\cG,C} \cong \res_{C/k}(\rB \cG)$. 
The stack $\bun_\cG$ is smooth and locally of finite type over $k$ (see e.g. \cite[Proposition 1.]{heinloth.uniformization}). 

If $\alpha \colon \calh \to \calh'$ is a morphism of affine algebraic groups, then for any $\calh$-bundle $\cale$ on $C$ we denote by $\alpha_{*}\cale$ the $\calh'$-bundle $\cale\times^{\calh} \calh'$. 
We remind the reader that for any morphism of $k$-schemes $f \colon D \to C$ we have a natural isomorphism 
\[
    f^* \alpha_* \cale \cong \alpha_* f^* \cale
\]
of $\calh'$-bundles on $D$.

Finally, given a stack $\mathcal{X}$, e.g. $\mathcal{X}=\bun_\calg$, and a notion of semistability on $\mathcal{X}$, we denote by $\mathcal{X}^{\Ss}$ the semistable locus.
\section{Moduli of principal bundles under nonconnected group schemes}\label{section:moduli}
\subsection{Decomposition of \texorpdfstring{$\bun_{\calg}$}{G}}\label{section:decomposition}
Fix a reductive group scheme $\calg$ over $C$. In this section we decompose $\bun_{\calg}$ into finitely many open and closed substacks $\bun_{\calp}$, each admitting a finite torsor $\bun_{\calg_{\calp}}\to \bun_{\calp}$ for some connected reductive group scheme $\calg_{\calp}$ over $C$. 
\begin{proposition}\label{proposition:bun-Gamma-discrete}
Let $\Gamma$ be a quasi-finite \'etale group scheme over $C$. If $\Char(k)>0$ assume furthermore that $\Gamma$ is finite and that $\Char(k)$ and $|\Gamma|$ are coprime. There is a finite decomposition 
  \[ \bun_{\Gamma,C} \cong \coprod_{[\calp] \in \rH_\et^1(C, \Gamma)} \rB \Aut(\calp). \]
\end{proposition}
\begin{proof}
Suppose first that $\Gamma$ is finite over $C$. Upon choosing a geometric base point $c$ in $C$, we have that $\rH_{\et}^1(C,\Gamma)=\rH^1(\pi_1^{\et}(C,c),\Gamma_c)$, where $\Gamma_c$ is the fiber of $\Gamma$ at $c$ equipped with its $\pi_1^{\et}(C,c)$-action. Since $\Gamma_c$ is finite and $\pi_1^\et(C,c)$ is (pro)finitely generated, we see that 
$\rH_{\et}^1(C,\Gamma)$ is finite. 

If $\Char(k)=0$ and $\Gamma$ is quasi-finite, then let $U\subset C$ be the largest open dense subset such that $\Gamma|_U$ is finite, and let $u$ be a geometric point in $U$. Then the functor
\begin{equation*}
    P \mapsto (P_u, \{ P_x \}_{x\in C\setminus U})
\end{equation*}
determines an equivalence of categories between the category of quasi-finite \'etale schemes $P\to C$ such that $P|_U$ is finite, and the category consisting of tuples $(\Sigma, \{F_x\}_{x\in C\setminus U})$ where $\Sigma$ is a finite $\pi_1^{\et}(U,u)$-set and $F_x \subset \Sigma^{\gal(k(C)^{\text{sep}}/K_x^{\text{h}})}$, where $K_x^{\text{h}}$ denotes the fraction field of the henselization of $\calo_{X,x}$. This functor respects products and takes group objects to group objects. The essential image of all $\Gamma$-bundles $P\to C$ thus consists of tuples $(\Sigma, \{F_x\}_{x\in C\setminus U})$ where $\Sigma$ is a $(\pi_1^{\et}(U,u),\Gamma_u)$-set and the $F_x$ are $\Gamma_x$-sets, where the action of $\Gamma_u$ on $\Sigma$ (resp. $\Gamma_x$ on $F_x$) is free and transitive. To give such a $(\pi_1^{\et}(U,u),\Gamma_u)$-set is equivalent to giving a cocycle in $\rH^1(\pi_1^{\et}(U,u),\Gamma_u)$. Since $\Char(k)=0$, $\pi_1^{\et}(U,u)$ is (pro)finitely generated and thus $\rH^1(\pi_1^{\et}(U,u),\Gamma_u)$ is finite.

Hence, in either case there are only finitely many isomorphism classes of $\Gamma$-bundles on $C$. 
  Since a $k$-valued point of $\bun_{\Gamma,C}$ is a $\Gamma$-bundle on $C$, we have a surjection
  \[ Y \coloneqq \coprod_{[\calp] \in \rH_\et^1(C, \Gamma)} \Spec{k} \to \bun_{\Gamma,C}. \]
  In particular, $\bun_{\Gamma,C}$ is quasi-compact.
  Moreover, since all stabilizer groups are finite, $\bun_{\Gamma,C}$ is a smooth Deligne-Mumford stack (see e.g. \cite[Theorem 3.6.1]{alper.stacksandmoduli}).  
  Thus, there exists an \'etale surjection $U \to \bun_{\Gamma,C}$ from a smooth, quasi-compact $k$-scheme $U$. 
  Consider the cartesian square
  \begin{center}
  \begin{tikzcd}
    V \arrow[r] \arrow[d] & U \arrow[d] \\
    Y \arrow[r] & \bun_{\Gamma,C}.
  \end{tikzcd}
  \end{center}
  Since $V \to Y$ is \'etale, the scheme $V$ is a disjoint union of copies of $\Spec{k}$, and since $V \to U$ is surjective, we see that $U$ is a finite disjoint union of copies of $\Spec{k}$ as well. Thus, $V \to U$ is \'etale, hence so is $Y \to \Bun_{\Gamma,C}$. The result follows.
\end{proof}
From now on, let $\Gamma\coloneqq \calg/\calg^\circ$, and let
\[ 
  1 \to \calg^\circ \xrightarrow{\iota} \calg \xrightarrow{\pi} \Gamma \to 1
\]
be the corresponding short exact sequence of affine algebraic group schemes over $C$.
The homomorphisms $\iota$ and $\pi$ induce
morphisms of stacks $\iota_*\colon \Bun_{\calg^\circ} \to \Bun_\calg$
and $\pi_*\colon \Bun_\calg \to \Bun_\Gamma$.
\begin{notation}\label{notation:BunP}
For any $\Gamma$-torsor $\calp$ over $C$, we define the stack $\Bun_\calp$ to be the preimage of the residual gerbe $\rB \Aut(\calp)\hookrightarrow \Bun_\Gamma$ corresponding to $\calp$. In other words, we have a cartesian diagram
\begin{equation*}
\begin{tikzcd}
  \Bun_\calp \arrow[r] \arrow[d] & \Bun_\calg \arrow[d, "\pi_{*}"] \\
  \rB\Aut(\calp) \arrow[r, "\calp"] & \Bun_{\Gamma}
\end{tikzcd}
\end{equation*}
where the horizontal maps are open and closed embeddings.
\end{notation}
%
%
%
%
%
%
%
\begin{lemma}\label{bunp.admits.finite.torsor}
Let $\calp$ be a $\Gamma$-torsor over $C$ admitting a lift to a $\calg$-torsor. There is a connected reductive group scheme $\mathcal{G}_\calp$ over $C$ and an $\Aut(\calp)$-torsor $\Bun_{\mathcal{G}_\calp} \to \Bun_\calp$.
\end{lemma}
\begin{proof}
Suppose first that $\calp \cong \bOne_\Gamma$ is the trivial $\Gamma$-torsor. The short exact sequence
\[ 1 \to \calg^\circ \xrightarrow{\iota} \calg \xrightarrow{\pi} \Gamma \to 1 \]
induces a cartesian diagram of stacks over $C$
\begin{equation}\label{diagram:BG-Speck-over-BGamma}
    \begin{tikzcd}
        \rB\calg^\circ  \arrow[r, "\iota_{*}"] \arrow[d] & \rB \calg \arrow[d, "\pi_{*}"] \\
        C \arrow[r, "\bOne_{\Gamma}"] & \rB\Gamma
    \end{tikzcd}
\end{equation}
Since restricting scalars from $C$ to $\Spec{k}$ preserves fiber products,
applying 
$\res_{C/k}$ to the diagram above 
we obtain a cartesian square
\begin{equation*}
    \begin{tikzcd}
        \bun_{\cG^\circ} \arrow[r, "\iota_{*}"] \arrow[d] & \bun_{\calg} \arrow[d, "\pi_{*}"] \\
        \Spec k \arrow[r, "\bOne_\Gamma"] & \bun_{\Gamma}
    \end{tikzcd}
\end{equation*}
This is readily seen to sit inside the larger commutative diagram
\begin{equation*}
    \begin{tikzcd}
        \bun_{\calg^{\circ}} \arrow[r, "\iota_{*}"] \arrow[d] & \bun_{\bOne_\Gamma} \arrow[r] \arrow[d] & \bun_{\cG} \arrow[d, "\pi_{*}"] \\
        \Spec k \arrow[r] & \rB\Gamma \arrow[r, "\bOne_\Gamma"] & \bun_{\Gamma}.
    \end{tikzcd}
\end{equation*}
Since the outer square and the right hand square are cartesian, so is the left hand square. Hence, $\iota_* \colon \bun_{\cG^{\circ}} \to \bun_{\bOne_\Gamma}$ is a $\Gamma$-torsor.

Let now $\calp$ be an arbitrary $\Gamma$-torsor over $C$ admitting a lift to a $\calg$-torsor $\cale$. 
Using \cite[\href{https://stacks.math.columbia.edu/tag/06QF}{Lemma 06QF}]{stacks-project}, we deduce from \eqref{diagram:BG-Speck-over-BGamma} 
that the morphism $\rB \calg \to \rB \Gamma$ is a gerbe, and so by \cite[\href{https://stacks.math.columbia.edu/tag/06QE}{Lemma 06QE}]{stacks-project}, the stack $\mathcal{Z}$ defined by the cartesian square
\begin{equation}\label{stack.of.lifts.of.Pi}
    \begin{tikzcd}
        \mathcal{Z} \arrow[r] \arrow[d] & \rB \calg \arrow[d, "\pi_{*}"] \\
        C \arrow[r, "\calp"'] & \rB\Gamma
    \end{tikzcd}
\end{equation}
is also a gerbe.
The assumption that $\calp$ has a lift to a $\calg$-torsor
implies that $\mathcal{Z} \to C$ admits a section. 
Hence, by \cite[\href{https://stacks.math.columbia.edu/tag/06QG}{Lemma 06QG}]{stacks-project}, 
we see that $\mathcal{Z} \cong \rB \calg_\calp$,
where $\calg_\calp$ is the subgroup scheme of $\uAut(\cale)$ whose sections are morphisms $\alpha$ of $\cale$ such that $\pi_* \alpha = \id_{\calp}$. 
Since $\calp$ is \'etale locally isomorphic to $\bOne_\Gamma$,
we see that $\calg_\calp$ is \'etale locally isomorphic to $\calg^\circ$,
and hence $\calg_\calp$ is connected and reductive.
From diagram (\ref{stack.of.lifts.of.Pi}) we see that for any $k$-scheme $T$,
the groupoid $\rB \calg_\calp(C\times T)$ consists of the data of
a $\calg$-torsor $\calf$ over $C\times T$
together with an isomorphism $\alpha\colon \pr_1^{*}\calp \to \pi_{*}\calf$,
where $\pr_1\colon C\times T\to C$ is the projection.
Hence, we see that
\begin{equation*}\label{main.diagram.bung.decomp}
    \bun_{\calg_{\calp},C} \cong \res_{C/k}(\rB \calg_\calp) \cong \Spec k \times_{\calp,\bun_{\Gamma,C},\pi_{*}} \bun_{\calg,C}.
\end{equation*}
Since the morphism $\bun_{\calg_{\calp},C} \to \bun_{\calg,C}$
is identified with the forgetful morphism $(\calf, \alpha) \mapsto \calf$, we see that it factors through $\Bun_\calp$. That is, we have a commutative diagram 
\begin{equation*}
    \begin{tikzcd}
        \bun_{\calg_{\calp},C} \arrow[r, "\psi_{\calp}"] \arrow[d] & \bun_{\calp} \arrow[r] \arrow[d] & \bun_{\cG,C} \arrow[d, "\pi_{*}"] \\
        \Spec k \arrow[r] & \rB \Aut(\calp) \arrow[r, "\calp"] & \bun_{\Gamma,C}.
    \end{tikzcd}
\end{equation*}
Since the right and outer squares are cartesian, so is the left. Since $\Gamma$ is finite, so is $\Aut(\calp)$, and since $\Spec k \to \rB \Aut(\calp)$ is an $\Aut(\calp)$-torsor, so is $\psi_{\calp}$. 
\end{proof}
%
%
%
\begin{theorem}\label{bung.decomp}
Let $\calg$ be a reductive group scheme over $C$. If $\Char(k)>0$ assume furthermore that $\Gamma$ is finite and that $\Char(k)$ and $|\Gamma|$ are coprime. Then there exist a finite decomposition $\bun_{\calg} \cong \coprod_{\calp}\bun_{\calp}$ by open and closed substacks such that each $\bun_{\calp}$ admits a torsor $\bun_{\calg_\calp} \to \bun_{\calp}$ under a finite group, where $\calg_{\calp}$ is a connected reductive group scheme over $C$.
\end{theorem}
\begin{proof}
By \Cref{proposition:bun-Gamma-discrete} we have a finite decomposition
\begin{equation*}
    \bun_{\calg} \cong \coprod_{\rH^1_{\et}(C,\Gamma)} \bun_{\calp}.
\end{equation*}
By \Cref{bunp.admits.finite.torsor} the statement follows. 
\end{proof}
\begin{remark}
\Cref{bung.decomp} allows us to reduce any question about nonconnected groups to the connected case, as long as we have a good theory of descent along finite torsors. 
In \Cref{section:existence.of.moduli} we use it to construct a good moduli space for the semistable locus. Thus, we first prove that existence of projective good moduli spaces descend along finite torsors.
\end{remark}
\begin{notation}
From now on, we assume that $\Char(k)=0$.
\end{notation}
\begin{lemma}\label{gamma.descent.gms}
Let $A$ be a finite group, let $\calx$ and $\caly$ be quasi-compact algebraic stacks over $k$, locally of finite type with affine diagonal
, and let $\phi\colon \calx \to \caly$ be an $A$-torsor. If $\calx$ admits a separated (resp. proper, resp. projective) good moduli space, then so does $\caly$. 
\end{lemma}
\begin{proof}
Suppose that $\calx$ admits a separated good moduli space $X$. Then $\calx$ is both $\Theta$-reductive and S-complete by \cite[Theorem A]{AHLH}. Since $\phi\colon \calx\to \caly$ is an $A$-torsor, it is in particular finite, \'etale, and surjective. Thus, by \cite[Propositions 3.19(2) and 3.43(2)]{AHLH} the stack $\caly$ is also $\Theta$-reductive and S-complete, and so by \cite[Theorem A]{AHLH}, it admits a separated good moduli space $Y$. 

If $X$ is proper, then since the surjection $\phi\colon \calx\to \caly$ induces a surjection $X \to Y$, we see that $Y$ is proper as well.

Suppose now that $X$ is projective and let $L$ be an ample line bundle on it. By the universal property of $X$ \cite[Theorem 6.6]{alper-gms}, the action of $A$ on $\calx$ descends to $X$. We can replace $L$ by $\otimes_{\gamma \in A} (\gamma^* L)$ and assume that $L$ is $A$-invariant. 
Consider the commutative diagram
\begin{equation*}
    \begin{tikzcd}
        \calx \arrow[r, "\phi"] \arrow[d, "f"] & \caly \arrow[d, "g"] \\
        X \arrow[r, "\pi"] & Y.
    \end{tikzcd}
\end{equation*}
Since $f^*L$ is invariant under the action of $A$, it descends to a line bundle $M$ on $\caly$.

We now show that $M^{\otimes n}$ descends to $Y$, where $n \coloneqq |A|$. From the Cartesian diagram
\begin{equation*}
    \begin{tikzcd}
        \calx \arrow[r] \arrow[d, "\phi"] & \Spec k \arrow[d] \\
        \caly \arrow[r] & \rB A
    \end{tikzcd}
\end{equation*}
we see that for any $k$-points $y \in \caly$ and $x \in \calx$ such that $\phi(x) = y$, the index of $\phi(\Aut(x))$ in $\Aut(y)$ is at most $n$. Now $\Aut(x)$ acts trivially on the fiber of $f^*L$ at $x$, and so we see that $\phi(\Aut(x))$ is contained in the kernel of the action of $\Aut(y)$ on the fiber of $M$ at $y$, hence $\Aut(y)$ acts trivially on $M^{\otimes n}$. Thus by \cite[Theorem 10.3]{alper-gms}, $M^{\otimes n}$ descends to a line bundle $N$ on $Y$. Now 
\[ f^* \pi^* N \cong \phi^* g^* N \cong f^* L^{\otimes n}, \]
so it follows that $\pi^* N \cong L^{\otimes n}$.
Since $\phi$ is affine, so is $\pi$, and since $\pi$ is also proper, it is finite, and so $N$ is ample on $Y$ by \cite[\href{https://stacks.math.columbia.edu/tag/0GFB}{Lemma 0GFB}]{stacks-project}. 
\end{proof}
%
%
%
%
%
%
%
%
%
%
%
%
%
%
%
%
%
\subsection{Semistability}\label{section:semistability}
In \cite{raman1, raman2} Ramanathan defined a notion of semistability for $G$-bundles under a  connected reductive group $G$ over $k$. When $G = \GL_n$, this notion agrees with $\mu$-semistability when $\GL_n$-bundles are viewed as vector bundles of rank $n$ (see \cite[Corollary 1]{hyeon-murphy} for a proof), and Ramanathan proved that a $G$-bundle $E$ is semistable if and only if $\Ad_{*}E$ is semistable (see \cite[Corollary 3.18]{raman1}). 

Let now $\calg$ be a connected reductive group scheme over $C$. In \cite{behrend} Behrend introduced the following notion of semistability for $\calg$-bundles, which reduces to Ramanathan's notion in the case of a constant group scheme. A $\calg$-bundle $\cale$ is \textit{semistable} if for each parabolic subgroup $\calp \subset \uAut(\cale)$, we have that
\begin{equation*}
    \deg \text{Lie}(\calp) \leq 0.
\end{equation*}
\begin{lemma}\label{semistability.via.adjoint.rep}
If $\calg$ is a connected reductive group over $C$, then a $\calg$–bundle $\cale$ is semistable if and only if $\Ad_*\cale$ is semistable.
\end{lemma}
\begin{proof}
Assume first that $\calg$ is an inner form of a constant connected reductive group $G$. That is, there is a $G$-bundle $E$ such that $\calg \cong \uAut(E)$. The functor $\cale \mapsto \cale \times^{\calg}E$ is then an equivalence of categories between $\calg$-bundles on $C$ and $G$-bundles on $C$. Under this equivalence semistable $\calg$-bundles correspond to semistable $G$-bundles. Indeed, if $\cale$ is a $\calg$-bundle, then we have that
\begin{equation*}
\uAut(\cale)\cong\cale\times^{\calg} \calg \cong \cale \times^{\calg} (E\times^{G} G) \cong (\cale \times^{\calg} E)\times^{G} G \cong \uAut(\cale\times^{\calg} E).
\end{equation*}
Furthermore, since $\Lie(\calg) \cong E\times^G \Lie(G)$ we see that for any $\calg$-bundle $\cale$ we have that
\begin{equation*}
  \Ad_*(\cale)\cong \cale\times^\calg \Lie(\calg) \cong \cale \times^{\calg} E \times^{G} \Lie(G) \cong \Ad_{*}(\cale \times^{G} E).  
\end{equation*}
Since $\cale$ is semistable if and only if $\cale \times^{G} E$ is semistable, the statement follows from the corresponding statement for a constant group.

In the general case, by \cite[Section 4.2]{behrend}, there is a finite \'etale cover $f\colon D\to C$ such that $f^*\calg$ is an inner form of a constant group. By \cite[Corollary 7.4]{behrend} any $\calg$-bundle $\cale$ is semistable if and only if $f^*\cale$ is semistable, and since $\Ad_*f^*\cale \cong f^*\Ad_*\cale$ the statement thus follows.
\end{proof}
We thus make the following definition of semistability for $\calg$–bundles to include the nonconnected case.
\begin{definition}
Let $\calg$ be a reductive group scheme over $C$. A $\calg$-bundle $\cale$ is \textit{semistable} if $\Ad_{*}\cale$ is semistable in the sense above.
\end{definition}
From the previous discussions we immediately obtain the following statements.
\begin{lemma}\label{ss.iff.ss.on.finite.cover}
Let $\calg$ be a reductive group scheme over $C$ and let $\cale$ be a $\calg$-bundle. The following are equivalent.
\begin{enumerate}
    \item The bundle $\cale$ is semistable;
    \item for any finite cover $f\colon D\to C$, $f^*\cale$ is semistable; and
    \item there exists a finite cover $f\colon D\to C$ such that $f^*\cale$ is semistable.
\end{enumerate}
\end{lemma}
\begin{proof}
This is true applied to the $\GL(\text{Lie}(\calg))$-bundle $\Ad_*\cale$. Since $\Ad_* f^* \cale \cong f^*\Ad_* \cale$ we see that the statement follows.
\end{proof}
\begin{lemma}\label{g0.ss.iff.iotag.ss}
Let $\iota\colon \calg^\circ \to \calg$ denote the inclusion. A $\calg^\circ$-bundle $\cale$ is semistable if and only $\iota_*\cale$ is a semistable $\calg$–bundle.
\end{lemma}
\begin{proof}
This follows from \Cref{semistability.via.adjoint.rep}.
%
\end{proof}
We denote by $\bun_{\calg}^{\Ss}$ the substack of semistable $\calg$-bundles. Since it is the preimage of $\bun_{\GL(\text{Lie}(\calg))}^{\Ss}$ under $\Ad_{*}$, it is an open substack. For any $\Gamma$-torsor $\calp$ admitting a lift to $\calg$, we similarly let $\bun_{\calp}^{\Ss}$ denote the substack of $\bun_{\calp}$ consisting of semistable bundles (recall \cref{notation:BunP}).
%
%
%
%
%
%
%
\subsection{Existence of a projective good moduli space}\label{section:existence.of.moduli}
We recall the following theorem of Behrend.
\begin{lemma}\label{passing.to.cover.behrend}(\cite[Theorem 4.4.3]{behrend.thesis}
Let $f\colon D\to C$ be a projective flat cover of $C$, and let $\calg$ be an affine group scheme of finite type over $C$. The induced morphism 
\begin{equation*}
    f^* \colon \bun_{\calg,C}\to \bun_{f^*\calg,D}
\end{equation*}
is affine of finite presentation.
\end{lemma}
We deduce immediately the following.
\begin{proposition}\label{pullback.is.affine.on.moduli.space}
Let $\calg$ be an affine smooth group scheme of finite type over $C$ and suppose that there is a notion of semistability on $\bun_{\calg}$ for which the substack of semistable bundles admits a proper good moduli space $M_{\calg,C}^{\Ss}$. For any finite cover $f\colon D\to C$ such that a $\calg$–bundle $\cale$ is semistable if and only if $f^*\cale$ is semistable, $f$ induces via pullback a finite morphism $f^*\colon M_{\calg,C}^{\Ss}\to M_{f^*\calg,D}^{\Ss}$.
\end{proposition}
\begin{proof}
By assumption we have that $\bun_{\calg,C}^{\Ss} = \bun_{\calg,C}\times_{\bun_{f^*\calg,D}}\bun_{f^*\calg,D}^{\Ss}$. Hence, by \Cref{passing.to.cover.behrend} $f^*$ induces an affine morphism between the semistable loci, and thus also induces an affine morphism between the good moduli spaces $f^*\colon M_{\calg,C}^{\Ss}\to M_{f^*\calg,D}^{\Ss}$. Since $M_{\calg,C}^{\Ss}$ and $M_{f^*\calg,D}^{\Ss}$ are proper, so is $f^*$. Hence, $f^*\colon M_{\calg,C}^{\Ss}\to M_{f^*\calg,D}^{\Ss}$ is finite.
\end{proof}
\begin{lemma}\label{ahlh.projective}
Let $\calg$ be a connected reductive group scheme over $C$. 
The stack $\bun_{\calg}^{\Ss}$ admits a projective good moduli space.
\end{lemma}
\begin{proof}
By \cite[Section 8]{AHLH} $\bun_{\calg}^{\Ss}$ admits a proper good moduli space $M_{\calg}^{\Ss}$. Let $f\colon D\to C$ be a finite \'etale cover such that $f^*\calg$ is an inner form of a constant group scheme $G$. By  \Cref{pullback.is.affine.on.moduli.space} we obtain a finite morphism 
\begin{equation*}
    f^* \colon M_{\calg,C}^{\Ss}\to M_{f^*\calg,D}^{\Ss}
\end{equation*}
between proper good moduli spaces. Since $f^*\calg$ is an inner form of $G$, $M_{f^*\calg,D}^{\Ss}\cong M_{G,D}^{\Ss}$, which is projective (\cite{raman1, raman2}). Since $f^*$ is finite this implies that $M_{\calg,C}^{\Ss}$ is projective.
%
\end{proof}
%
%
%
%
%
%
%
%
%
%
\begin{theorem}\label{main.theorem}
Let $\calg$ be a (not necessarily connected) reductive group scheme over $C$. 
The semistable locus $\Bun_{\calg}^\Ss$ admits a projective good moduli space.
\end{theorem}
\begin{proof}
By \Cref{bung.decomp} it suffices to prove that for any $\Gamma$-torsor $\calp$ admitting a lift to a $\calg$-torsor, the semistable locus $\bun_\calp^{\Ss}$ admits a projective good moduli space. To this end, let $\calp$ be a $\Gamma$-torsor  admitting a lift to a $\calg$-torsor. Since $\rH_{\et}^1(C,\Gamma)$ is finite (cf. proof of \Cref{proposition:bun-Gamma-discrete}), there exists a quasi-finite \'etale cover $C'\to C$ trivialising all $\Gamma$-torsors. By restricting to an irreducible component we may assume that $C'$ is irreducible. Let $f\colon D\to C$ be the finite cover of $C$ corresponding to the field extension $k(C')/k(C)$. Then $f$ factors through a rational morphism $D\dasharrow C' \to C$. Hence, for any $\Gamma$-bundle $\calp$ over $C$, $f^*\calp$ is generically trivial, whence trivial (by \cite[Proposition 3.4]{antei.emsalem.gasbarri}). 
The short exact sequence
\begin{equation*}
    1\to \calg^\circ \to \calg \to \Gamma \to 1
\end{equation*}
induces a long exact sequence of pointed sets
\begin{equation*}
    \cdots \to \rH_{\et}^1(C,\calg^\circ) \to \rH_{\et}^1(C,\calg) \to \rH_{\et}^1(C,\Gamma)
\end{equation*}
which is functorial with respect to $f^*$. Hence, for any $\calg$–bundle $\calf$, $f^*\calf$ has a reduction to $\calg^\circ$. We thus have a commutative diagram
\begin{equation*}
\begin{tikzcd}
    \bun_{\calg^\circ,D} \arrow[r, "\iota_{*}"] & \bun_{\calg,D} \\
    \bun_{\calg_\calp,C} \arrow[r, "\psi_{\calp}"] \arrow[u, "f^{*}"]  & \bun_{\calp,C} \arrow[u, "f^{*}"] 
\end{tikzcd}
\end{equation*} 
Let now $\cale$ be a $\calg_\calp$-torsor. By \Cref{ss.iff.ss.on.finite.cover}  
we see that $\cale$ is semistable if and only if $f^{*}\cale$ is semistable, which by \Cref{g0.ss.iff.iotag.ss} is equivalent to $\iota_{*}f^{*}\cale\cong f^{*}\psi_{\calp}(\cale)$ being semistable. Again using \Cref{ss.iff.ss.on.finite.cover} we thus see that $\cale$ is semistable if and only if $\psi_{\calp}(\cale)$ is semistable. 
Hence, by base change along $\bun_{\calp,C}^{\Ss}\hookrightarrow \bun_{\calp,C}$, the morphism $\psi_{\calp}$ induces an $\Aut(\calp)$-torsor on the semistable loci
\begin{equation*}
   \psi_{\calp} \colon \bun_{\calg_\calp,C}^{\Ss} \to \bun_{\calp,C}^{\Ss}.
\end{equation*}
By \Cref{ahlh.projective}, $\bun_{\calg_\calp,C}^{\Ss}$ admits a 
projective good moduli space. 
Hence, by Lemma \ref{gamma.descent.gms} we conclude that $\bun_{\calp,C}^{\Ss}$ admits a projective good moduli space.
\end{proof}
We remark also an immediate consequence of \Cref{pullback.is.affine.on.moduli.space} and \Cref{main.theorem} that has appeared in the literature in the case of vector bundles (see \cite[Theorem 4.2]{weissmann} cf. \cite[Theorem 4.2]{hein}, \cite[Theorem I.4]{faltings}) and could be of independent interest.
\begin{corollary}\label{corollary:pullback.is.finite}
Let $\calg$ be a reductive group scheme over $C$. For any finite cover $f\colon D\to C$, the induced morphism $f^{*}\colon M_{\calg,C}^{\Ss} \to M_{f^*\calg,D}^\Ss$ 
is finite.
\end{corollary}
\begin{proof}
By \Cref{main.theorem} and \Cref{ss.iff.ss.on.finite.cover} this is a special case of \Cref{pullback.is.affine.on.moduli.space}.
\end{proof}
\begin{remark}
\Cref{corollary:pullback.is.finite} applies to any affine  group scheme $\calg$ such that $\bun_{\calg}^{\Ss}$ admits a proper good moduli space, and such that semistability can be checked on a finite cover. In particular, it applies to any Bruhat-Tits group scheme of the form $\calg = \res_{C'/C}(\mathcal{G}')^{\gal(C'/C)}$ for some Galois cover $f\colon C'\to C$ and some connected reductive group $\calg'$ over $C'$. 
\end{remark}
%
%
%
%
%
%
%
%
\section{Nondensity of essentially finite bundles}\label{section:essentially.finite.bundles}
\subsection{Finite subgroups of reductive groups}\label{section:finite.subgroups} 
We remind the reader that we continue with the assumption that $\Char(k)=0$. 

Let $G$ be an arbitrary reductive group over $k$. The main purpose of this section is to prove that there are finitely many subgroups $H_1, ..., H_n$ of $G$ such that each $H_i^\circ$ is a torus and any finite subgroup of $G$ is contained in some $H_i$, up to conjugation. To do so, recall first Jordan's classical result in \cite{jordan} on finite subgroups of $\GL(n)$.
\begin{theorem}\label{Jordan}(\cite{jordan})
For each integer $n>1$ there is an integer $f_n$ such that every finite subgroup $\Gamma \subset \GL(n)$ contains an abelian normal subgroup $A$ such that $[\Gamma:A] \leq f_n.$
\end{theorem}
Inspired by this result, we prove the following generalized version.
\begin{theorem}\label{general.jordan.intext}
Let $G$ be a (not necessarily connected) reductive group. There exists a constant $f_G \geq 1$ with the following property: for every finite subgroup $\Gamma \subset G$ there is an abelian normal subgroup $A \subset \Gamma$ contained in a torus of $G^{\circ}$ such that $[\Gamma:A] \leq f_G$. 
\end{theorem} 
The proof of \Cref{general.jordan.intext} closely follows Breuillard's exposition in \cite{breuillard.jordan.exposition} on Jordan's theorem. Mainly, we make sure that one can run the same arguments upon making appropriate changes from the $\GL(n)$-case to the general case of a reductive group (e.g. the use of the center instead of scalar matrices, the use of the projection $G\to G/[G,G]$ instead of the determinant map, etc.). 
As such, we 
relegate the proof to an appendix (see \Cref{appendix}). 

Before proving the main result of this section, we prove the following lemma and introduce some notation.
\begin{lemma}\label{finite.conjugacy}
For a (not necessarily connected) reductive group $G$ and a finite subgroup $\Gamma \subset G$, there is a finite number of elements in $\Hom(\Gamma, G)$ up to conjugacy of $G$. 
\end{lemma}
\begin{proof}
    We can give $\Hom(\Gamma,G)$ the structure of a variety over $k$, and we can consider the categorical quotient quotient $X(\Gamma,G) \coloneqq \Hom(\Gamma,G)//G$, the $G$-character variety of $\Gamma$. It is an affine variety, and the points of $X(\Gamma,G)$ are completely reducible representations of $\Gamma$ up to conjugacy. If $\Gamma$ is finite and $\Char(k)=0$, this is just every homomorphism up to conjugacy. If $\rho:\Gamma \to G$ is a homomorphism, then
    \[\dim T_{[\rho]}X(\Gamma,G) \leq H^1(\Gamma,\mathfrak{g}), 
    \]
    where $T_{[\rho]}$ is the tangent space at the class $[\rho]\in X(\Gamma,G)$ and where $\Gamma$ acts on $\mathfrak{g}$ trough $\rho$. This follows from \cite[Theorem 1]{sikora}. Take a cocycle $\gamma \in H^1(\Gamma,\mathfrak{g})$, and define for $g \in G$ the map $\varphi_g:\mathfrak{g} \to \mathfrak{g}$
    \[\varphi_g(X)=\rho(g)X+\gamma(g).\]
    This is an action of $\Gamma$ on $\mathfrak{g}$ by affine transformations. By averaging over an orbit, we find a fixed point $X_0$. We therefore have
    \[\gamma(g)=\varphi_g(X_0)-\rho(g)X_0=X_0-\rho(g)X_0.\]
    In other words, $\gamma$ is a coboundary and $H^1(\Gamma,\mathfrak{g})=0$. We see that $X(\Gamma,G)$ is finite.
\end{proof}
Given a subset $S\subset G$ we denote by $C_G(S)$ and $N_G(S)$ the centralizer and the normalizer of $S$ in $G$. 
Fix a maximal torus $T\subset G$. For any representation $V$ of $G$ with central kernel, let $\Phi(V)$ be the set of weights of $V$ relative to $T$. We say that a subgroup $H \subset T$ is a \textit{weight kernel} if it is defined by a finite number of equations of the form $\alpha(t)=\beta(t)$ for $\alpha,\beta \in \Phi(V)$ and $t \in T$. That is, for any $S\subset \Phi^2$ we have a weight kernel
\begin{equation*}
    H_S = \{ \; t \in T \; : \; \alpha(t)/\beta(t) = 1 \; \text{for all} \; (\alpha,\beta) \in S \; \}. 
\end{equation*}
Conversely, any weight kernel gives a subset $S\subset \Phi^2$. We see that
\begin{equation*}
    S \subset S' \iff H_S \supset H_{S'}.
\end{equation*}
If $A \subset T$ is a group, then we let $H_A$ be the minimal weight kernel containing $A$. Since both $C_G(A)$ and $C_G(H_A)$ consist exactly of the elements preserving the decomposition $V$ into characters of $A$, we see that $C_G(A)=C_G(H_A)$. Similarly, since both $N_G(A)$ and $N_G(H_A)$ consists exactly of the elements permuting the characters of $V$ we obtain that $N_G(A)=N_G(H_A)$.
\begin{theorem}\label{main.theorem.on.finite.groups}
    Let $G$ be a reductive group. There is a set $S(G)= \{H_1, \ldots, H_n\}$ of closed subgroups of $G$ such that $H_i^\circ$ is a torus and for any finite subgroup $\Gamma \subset G$ some conjugate is contained in one of the $H_i$.
\end{theorem}
\begin{proof}
We proceed by induction on the dimension of $G$. If $\dim G = 1$, then $G^{\circ}$ is a torus, so we are done. For the induction step, we assume the statement of the theorem for any group reductive group $H$ such that $\dim H < \dim G$. Let $\Gamma \subset G$ be a finite subgroup. By \Cref{general.jordan.intext}, there is an abelian normal subgroup $A \subset \Gamma$ contained in a maximal torus $T \subset G$. 

First, assume that $A$ is not in the center of $G^{\circ}$. This means that $C_G(A) \subsetneq G$, which implies that $\dim C_G(A) < \dim G$. Since $C_G(A)$ has finite index in $N_G(A)$, we also see that $\dim N_G(A)< \dim G$. Now $\Gamma \subset N_G(A)$. Up to conjugacy, the number of groups of the form $N_G(A)$ as we let $\Gamma$ run over all the finite subgroups of $G$ will be finite. Indeed, we can choose a faithful representation $V$ of $G$ and a weight kernel $H_A$ relative $T$ such that $N_G(H_A)=N_G(A)$, and there are only finitely many such weight kernels. In other words, if we add $S(N_G(H))$ to $S(G)$ where $H$ is a weight kernel we get by induction that $\Gamma \subset H_i$ for some $H_i \in S(G)$. 

Suppose now that $A \subset Z(G^{\circ})$. By the structure theory of reductive groups, $G^{\circ}=ZG'$ where $Z \coloneqq Z(G^{\circ})^{\circ}$, $G'\coloneqq [G^{\circ},G^{\circ}]$ and where the intersection $Z \cap G'$ is finite. Consider the map $\rho:G \to G/Z$. The group $B\coloneqq Z((G/Z)^{\circ})=Z(G^{\circ})/Z$ is a finite group, and $\rho(A) \subset B$. This means that 
\begin{equation*}
    |\rho(\Gamma)|=[\rho(\Gamma):\rho(A)]|\rho(A)| \leq f_{G} |B|.
\end{equation*}
This means that $|\rho(\Gamma)|$ is bounded in terms of $G$. In particular, up to isomorphism the set 
\begin{equation*}
    \{ \rho(\Gamma) : \Gamma \subset G, A\subset \rZ(\Go) \}
\end{equation*}
is finite. By \Cref{finite.conjugacy}, for a given choice of $\Gamma$ there is a finite number of conjugacy classes in $\Hom(\rho(\Gamma),G/Z)$. This means there is a finite set of choices for $\rho(\Gamma)$. Each choice determines a group $H=\rho^{-1}(\rho(\Gamma))$, which has connected component a torus. We add each $H$ to $S(G)$. 
\end{proof}
\subsection{Nondensity of essentially finite bundles}\label{section:nondensity}
In this section, $G$ is a connected reductive group over $k$. We recall that a $G$-bundle is essentially finite if it admits a reduction to a finite group. The same definition applies to nonconnected groups as well, and an important remark is that if a bundle is essentially finite, then there is a finite \'etale cover trivialising it (these notions are in fact equivalent but we won't need the other direction in what follows). We denote by $M_G^{\ef}$ the set of essentially finite $G$-bundles of degree 0 inside the moduli space of semistable $G$-bundles of degree 0, denoted $M_G^{\Ss,0}$. 
%
%

In this section we give an upper bound on the dimension of $\overline{M_G^{\ef}}$ and thus prove that $M_G^{\ef}$ is not dense in $M_G^{\Ss,0}$. The idea is as follows. Let $H_i$ be as in \cref{main.theorem.on.finite.groups} and denote by $\varphi_i \colon H_i\hookrightarrow G$ the inclusion. We would like to say that 
$M_G^{\ef}\subset \cup_i\varphi_{i,*}M_{H_i}^{\Ss}\subset M_G^{\Ss,0}$ and then conclude by dimension reasons. However, if $E$ is a semistable $H_i$-bundle, then $\varphi_{i,*}E$ need not be a semistable $G$-bundle, so we don't have a morphism $\varphi_{i,*}\colon M_{H_i}^\Ss\to M_G^\Ss$. Thus we prove first that there is a sublocus of semistable $H_i$-bundles whose pushforward is semistable. Then we show that every essentially finite $H_i$-bundle is contained in this sublocus. Finally we make sure that we have the expected relation between dimensions: $\dim M_{H_i}^\Ss < \dim M_G^\Ss$. After introducing some notation, we prove these claims in in a sequence of lemmas, leading up to the stated density result.

Let $\calt$ be a torus over $C$. By \cite{heinloth.uniformization} we have that
\begin{equation*}
    \pi_0(\Bun_\calt)=X_{*}(\calt_{k(C)^{\text{sep}}})_{\gal(k(C)^{\text{sep}}/k(C))}.
\end{equation*}
We denote by $\Bun_\calt^d$ the component of $\Bun_\calt$ corresponding to $d\in \pi_0(\Bun_\calt)$, and for $E\in \Bun_\calt^d(k)$, we say that $E$ is of degree $d$. We use similar notation for the semistable locus. By \cite{heinloth.hilbert-mumford} for all $d\in \pi_0(\Bun_\calt)$, $\bun_{\calt}^{\Ss,d}$ is of finite type. If $H$ is a non-connected reductive group such that $H^\circ=T$ is a torus, then let 
\[
\bun_H^0\coloneqq \im\Big(\coprod_{\calp \in \rH_\et^1(C,H/H^\circ)} \bun_{H_\calp}^{0} \to \bun_{H}\Big),
\]
with the notation $H_\calp$ analogous to that in \Cref{section:decomposition}. 
We define $\bun_H^{\Ss,0}$ and $M_H^{\Ss,0}$ similarly. 
%
\begin{lemma}\label{f.pullback.degree.mult} Let $\mathcal{T}$ be a torus over $C$, and let $f\colon C'\to C$ be a finite \'etale morphism. The induced map $f^{*}\colon \pi_0(\Bun_{\calt,C})\to \pi_0(\Bun_{\calt,C'})$ is given by multiplication with $\deg(f)$.
\end{lemma}
\begin{proof}
Suppose first that $\calt$ is induced, i.e. that there is some finite, generically \'etale morphism $\pi\colon D\to C$ such that $\calt \cong \pi_* \mg_m\coloneqq \res_{D/C}(\mg_m)$. The Leray spectral sequence
\begin{equation*}
    \rH^p_\et(C,\rR^q\pi_*\mg_m) \implies \rH^{p+q}_\et(D,\mg_m)
\end{equation*}
induces a filtration on $\rH^1_\et(D,\mg_m)$ whose graded pieces are $\rH^0_\et(C,\rR^1\pi_* \mg_m)$ and $\rH^1_\et(C,\pi_*\mg_m)$. For any point $x\in C$, by proper base change (see e.g. \cite{deninger}) and using that $\pi$ is finite, we see that $(\rR^1\pi_* \mg_m)_x \cong \rH^1_\et(D_x , \mg_m)=0$. Hence, $\rH^1_\et(C,\mathcal{T})\cong \rH^1_\et(D,\mg_m)$ via the Leray spectral sequence and the result follows from this known case for line bundles. 

In the general case, following the proof of \cite[Lemma 16]{heinloth.uniformization} (cf. \cite[Section 3. and 5.]{pappas.rapoport}) we can choose induced tori $\cali_2\to \cali_1\twoheadrightarrow \calt$ such that the induced sequence 
\begin{equation*}
    \pi_0(\Bun_{\cali_2})\to \pi_0(\Bun_{\cali_1})\to \pi_0(\Bun_\calt)\to 0
\end{equation*}
is exact. The result now follows.
\end{proof}
\begin{lemma}\label{ss.deg.0.to.ss}
Let $G$ be a connected, reductive group, let $T\subset G$ be a torus, and let $H\subset G$ be a subgroup such that $H^\circ = T$. 
The inclusion $\varphi \colon H\hookrightarrow G$ induces a morphism $\varphi_* \colon \bun_{H}^{\Ss,0}\to \bun_G^{\Ss}$.
\end{lemma}
\begin{proof}
We need to show that for any $H$-bundle $E$ of degree 0, $\varphi_*E$ is semistable. Let us first assume that $H=T$ and let $\rho \colon H\to \GL(n)$ be an arbitrary representation. The image $\im(\rho)$ is a torus in $\GL(n)$. Hence, since all tori are conjugate, we may assume that $\im(\rho)$ lies in $\diag_n$, the diagonal matrices. The induced map $T\to \diag_n$ takes $E$ to a degree 0 $\diag_n$-bundle. The vector bundle $\rho_*E$ can then be realised as $\rho_*E \cong L_1\oplus ... \oplus L_n$ for some line bundles $L_1, ..., L_n$, all of which are of degree 0. Hence $\rho_*E$ is semistable.

Now let $H$ be an arbitrary subgroup of $G$ such that $H^\circ =T$ is a torus and 
let $E$ be an $H$-bundle of degree 0, say coming from $\bun_{H_\calp}^{\Ss,0}$. Taking a sufficiently large finite \'etale cover $f\colon D\to C$, by \Cref{f.pullback.degree.mult} we have a commutative diagram
\begin{equation*}
    \begin{tikzcd}
        \bun_{T,D}^{\Ss,0} \arrow[r] & \bun_{H,D}^{\Ss,0} \\
        \bun_{H_\calp,C}^{\Ss,0} \arrow[r] \arrow[u, "f^*"] & \bun_{H,C}^{\Ss,0} \arrow[u, "f^*"]
    \end{tikzcd}
\end{equation*}
Since $\varphi_*E$ is semistable if and only if $\varphi_*f^*E$ is semistable, and since $f^*E$ has a reduction to a degree 0 $T$-bundle, upon checking at the representation $\rho\colon T\to H\to G\xrightarrow{\Ad} GL(\mathfrak{g})$ we see that $\varphi_*E$ is semistable.
\end{proof}
\begin{lemma}\label{ef.bundles.degree.0}
Let $H$ be a reductive group such that $H^\circ = T$ is a torus. Every essentially finite $H$-bundle has degree 0.
\end{lemma}
\begin{proof}
Suppose that $E$ is an essentially finite $H$-bundle. Using the notation as in Section \ref{section:existence.of.moduli} with $\calp\coloneqq \pi_*E$, we have a relative group torus $H_\calp$ over $C$ and a finite torsor $\bun_{H_\calp}\to \bun_{\calp}$. 
Let $\mathcal{E}$ be a $H_\calp$-torsor whose image in $\bun_{H}$ is $E$. 
We will show that the degree of $\mathcal{E}$ is 0. 

Let $\Gamma = H/H^\circ$ and let $f\colon C'\to C$ be a finite \'etale cover which trivialises all $\Gamma$-torsors. 
Since $E$ is essentially finite, so is $f^{*}E$ and hence, there is a finite \'etale cover $g\colon C''\to C'$ trivialising it. Letting $h\coloneqq f\circ g$ we obtain the following commutative diagram
\begin{equation*}
    \begin{tikzcd}
    \bun_{T,C''} \arrow[r, "\iota_{*}"] & \bun_{\bOne_\Gamma,C''} \\
        \bun_{H_\calp,C} \arrow[r] 
        \arrow[u,
        "h^*"] & \bun_{\calp,C}
        \arrow[u,
        "h^*"']
    \end{tikzcd}
\end{equation*}
where $h^{*}E \cong \bOne_{H}$. Hence, $\iota_{*}h^{*}\mathcal{E} \cong \bOne_{H}$. By Lemma \ref{f.pullback.degree.mult} 
we have that $\deg(h^{*}\mathcal{E})=\deg(h)\deg(\mathcal{E})$. 
To show that $\deg(\mathcal{E})=0$, it thus suffices to show that for any $T$-torsor $F$ over $C''$, if $\iota_{*}F\cong \bOne_{H}$, then $F$ has degree 0. To this end, suppose that $F$ is such a $T$-torsor, and let $(t_{ij})_{ij}$ be a cocycle representative, for some cover $(U_i\to C'')_i$. Since $T$ is commutative we obtain a $T$-torsor, denoted $F^2$, given in cocycles by $(t_{ij}^2)_{ij}$. Since $\iota_{*}F$ is trivial, there are $g_i\in H(U_i)$ such that $g_i t_{ij}g_{j}^{-1}=1$. Letting $\sim$ denote cohomological equivalence on cocycles, we then have that 
\begin{equation*}
    (t_{ij}^2)_{ij} \sim (g_i t_{ij}^2g_{j}^{-1})_{ij} 
    = (g_i t_{ij} g_{j}^{-1} g_j g_{i}^{-1} g_i t_{ij} g_{j}^{-1})_{ij} 
    = (g_j g_{i}^{-1})_{ij} 
    \sim (1)_{ij}
\end{equation*}
Thus, $\iota_{*}F^2 \cong \bOne_{H}$. Similarly, if $F^n$ denotes the $T$-torsor given in cocycles by $(t_{ij}^n)_{ij}$, for all $n\in \mathbb{N}$, then $\iota_{*}F^n\cong \bOne_{H}$. Since $\bun_{T,C''}\to \bun_{\iota,C''}$ is a $\Gamma$-torsor, the fibre over $\bOne_{H}$ is finite. Thus, for some $N\in \mathbb{N}$, $F^N\cong \bOne_T$. This implies that $F$ is essentially finite and hence of degree 0. Applying this to $F = h^*\mathcal{E}$ we see that $\deg(\mathcal{E})=0$.
\end{proof}
\begin{lemma}\label{dimension.arguments}
Let $g$ denote the genus of $C$ and assume that $g\geq 2$. Let $G$ be a connected, reductive group not equal to a torus, and let $H\subset G$ be a subgroup such that $H^\circ = T$, for some torus $T\subset G$. Then we have that
\begin{equation}\label{dimension.inequality}
    \dim M_{H}^{\Ss} \leq g\dim T.
\end{equation}
In particular, $\dim M_{H}^{\Ss} < \dim M_G^\Ss$.
\end{lemma}
\begin{proof}
Let $\overline{x}\in M_H^\Ss$ be a closed point, let $x\in \bun_H^\Ss$ be the unique closed point in the fiber over $\overline{x}$ and let $H_x$ denote the stabiliser of $x$. By the local structure theorem for good moduli spaces \cite[Corollary 6.5.3]{alper.stacksandmoduli} we have a Cartesian diagram
\begin{equation*}
    \begin{tikzcd}
        \left[ \Spec A/H_x \right] \arrow[r] \arrow[d] & \bun_H^\Ss \arrow[d] \\
        \Spec A^{H_x} \arrow[r] & M_H^\Ss
    \end{tikzcd}
\end{equation*}
where the horizontal maps are \'etale. Hence, $\dim M_H^\Ss = \dim \Spec A^{H_x}$ and $\dim \bun_H^\Ss = \dim \Spec A - \dim H_x$. 
To get an upper bound on $\dim \Spec A^{H_x}$ we may assume that $H_x$ is connected, and hence irreducible. Since $H_x$ is irreducible, it preserves the irreducible components of $\Spec A$ so we may furthermore assume that $\Spec A$, and thus $\Spec A^{H_x}$, are irreducible. Letting $d$ be the minimum of the dimensions of the stabilisers of points in $\Spec A$, then (see e.g. \cite[Chapter 6, Corollay 6.2]{dolgachev}) 
we have that 
\begin{equation*}
\begin{aligned}
    \dim \Spec A^{H_x} &= \dim \Spec A - \dim G + d \\
    &=\dim[\Spec A/H_x] + d \\
    &=\dim \bun_H^\Ss + d.
\end{aligned}
\end{equation*}
Since $d\leq \dim H_x \leq \dim H$, we thus see that 
\begin{equation}\label{dimension.of.moduli.space.torus}
\begin{aligned}
    \dim M_{H}^{\Ss} &\leq \dim \bun_{H}^{\Ss} + \dim H \\
    &= (g-1)\dim H + \dim H \\
    &= g\dim T.
\end{aligned}
\end{equation}
Since $\dim M_G^\Ss= (g-1)\dim G + \dim \rZ(G)$, to see that $\dim M_{H}^{\Ss} < \dim M_G^\Ss$ it suffices to show that
\begin{equation*}\label{equiv.dimension.inequality}
     \frac{\dim G - \dim \rZ(G)}{\dim G - \dim T} \leq g.
\end{equation*}
Without loss of generality, we may assume that $T$ is a maximal torus. Let $\Delta$ be a set of simple roots with respect to $T$. Since $\dim G = \dim \mathfrak{g}$ and $G$ is not a torus we see that $\dim G > \dim T + |\Delta|$. Since $\rZ(G) = \bigcap_{\alpha \in \Delta}\ker(\alpha)$ we see that $\dim T - \dim \rZ(G)\leq |\Delta|$ and we obtain that
\begin{equation*}
    \frac{\dim G - \dim \rZ(G)}{\dim G - \dim T} = 1 + \frac{\dim T - \dim \rZ(G)}{\dim G - \dim T}<1 + \frac{|\Delta|}{|\Delta|} =2 \leq g.
\end{equation*}
\end{proof}
\begin{theorem}
For any connected, reductive group $G$ not equal to a torus, we have that
\begin{equation}\label{main.theorem.ef.dimension.estimate}
    \dim \overline{M_G^{\textnormal{ef}}} \leq g\rk(G).
\end{equation}
In particular, $M_G^{\textnormal{ef}}\subseteq M_G^{\Ss,0}$ is not dense.
\end{theorem}
\begin{proof}
Let $H_1, ..., H_n\subset G$ be as in \Cref{main.theorem.on.finite.groups}, and let $\varphi_i\colon H_i \hookrightarrow G$ denote the inclusions. 
By \Cref{main.theorem.on.finite.groups} any essentially finite $G$-bundle has a reduction to some $H_i$, and by \cref{ss.deg.0.to.ss} we have a morphism $\varphi_{i,*}\colon M_{H_i}^{\Ss,0}\to M_G^\Ss$. By \Cref{ef.bundles.degree.0}, all essentially finite $H_i$-bundles are of degree 0, from which it follows that 
\begin{equation*}
    M_G^{\ef} \subset \bigcup_{i=1}^{n}\varphi_{i,*}M_{H_i}^{\Ss,0} \subset M_G^{\Ss,0}.
\end{equation*}
%
By (\ref{dimension.of.moduli.space.torus}) we have that $\dim M_{H_i}^{\Ss,0} \leq g\dim H_{i}^{0} \leq g\rk(G)$ for all $i$, from which we obtain the dimension estimate in (\ref{main.theorem.ef.dimension.estimate}). By \cref{dimension.arguments} we see that $M_{G}^{\ef,0}\subseteq M_{G}^{\Ss,0}$ is not dense.
\end{proof}
%
%
%
%
%
%
%
%
%
%
\appendix
\section{}\label{appendix}
The purpose of this appendix is to prove the following.
\begin{theorem}\label{general.jordan}
Let $G$ be a (not necessarily connected) reductive group and let $\Gamma \subset G$ be a finite subgroup. Then there is a constant $f_G \geq 1$ with the following property: For every finite subgroup $\Gamma \subset G$ there is an abelian, normal subgroup $A \subset \Gamma$ contained in a torus of $G^{\circ}$ such that $[\Gamma:A] \leq f_G$. 
\end{theorem} 
We follow \cite{breuillard.jordan.exposition}, only making necessary adjustments to make the theorem work for a general reductive group. 
\begin{definition}
Let $T$ be a torus. Let $\Gamma \subset G$ be a finite group and $M$ be a positive integer greater than 2. We say that a subgroup $F \subset \Gamma$ is an \textit{$M$-fan} if it is conjugate to a subgroup $F'$ of $T$ such that for all $\alpha,\beta \in \Phi$, $\alpha/\beta$ either attains $M$ different values on $F'$ or is reduced to 1 on $F'$.
\end{definition}
\begin{remark}
To clarify, any $M$-fan is by definition also an $M'$-fan for any $M'\leq M$. Further, it follows from the definition that if $F$ is an $M$-fan then $H_F$ is the weight kernel defined by setting the same $\alpha/\beta=1$ as those which reduce to 1 on $F$. We also have that $H_{H_A}=H_A$.
\end{remark}
\begin{lemma}\label{equivalence.of.fans.and.abelian.groups}
Let $G$ be a reductive group. The following are equivalent.
\begin{enumerate}
    \item There is a constant $f_G \geq 1$ with the following property: For every finite subgroup $\Gamma \subset G$ there is an abelian, normal subgroup $A \subset \Gamma$ contained in a torus of $G^{\circ}$ such that $[\Gamma:A] \leq f_G$. 
    \item There is a positive integer $N=N(G)$ such that for any positive integer $M>N$ any finite subgroup $\Gamma \subset G$ contains a unique maximal $M$-fan $\calf$ such that $[\Gamma:\calf] \leq N$.
\end{enumerate}
\end{lemma}
\begin{proof}
Let $\Gamma \subset G$ be an arbitrary finite subgroup.

If (2) holds, then $\Gamma$ contains a unique maximal $M$-fan, which is normal and toral. Hence, with $f_G\coloneqq N$ we obtain (1).

Conversely, assume that (1) holds. Let $N\coloneqq f_G$, and let $A\subset \Gamma$ be a normal toral subgroup such that $[\Gamma:A] \leq N$. Then $\Gamma \subset N_G(A)=N_G(H_A)$, where $H_A$ is a weight kernel for the adjoint representation. We let $M>N$. If $F$ is an $M$-fan in $\Gamma$ and $f \in F$, then we have that $f^N \in A \subset H_A$ since $[F:F \cap A] \leq [\Gamma:A] \leq N$. Hence, for any $\alpha_1,\alpha_2 \in \Phi(\fkg)$, if $\alpha_1/\alpha_2$ attains $M$ different values on $F$, then $(\alpha_1/\alpha_2)(f^N)$ attains at least two different values as we go trough $f \in F$. Furthermore, since $M>N$, $\alpha_1/\alpha_2$ reduces to 1 on $F$ if and only if it does so on $F\cap H_A$. This means that the weights determining $F \cap H_A$ as a 2-fan in $\Gamma$ are the same as the weights determining $F$ as an $M$-fan in $\Gamma$. Hence, $H_{F\cap H_A} = H_F$, which means that 
\begin{equation*}
    F \subset H_F=H_{F\cap H_A} \subset H_A.
\end{equation*}
Since $F$ was arbitrary this implies that any $M$-fan in $\Gamma$ is contained in $H_A$. Let $\calf$ be the group generated by all $M$-fans in $\Gamma$. This is an $M$-fan. Since $A\subset \Gamma \cap H_A$ we thus obtain that
\begin{equation*}
    [\Gamma : \calf]\leq [\Gamma:\Gamma \cap H_A] \leq [\Gamma : A] \leq N.
\end{equation*}
This means that (2) follows with $\calf=\Gamma \cap H_A$.
\end{proof}
\subsection{Proof of \Cref{general.jordan}}
We proceed by induction on $\dim G$. If $\dim G=1$, then $\Go$ is a torus which means that we can choose $A\coloneqq \Gamma \cap \Go$. 

We now let $G$ be an arbitrary reductive group with $\dim G >1$ and we assume \Cref{general.jordan} for all reductive groups $H$ such that $\dim H < \dim G$. We keep this assumption for the remainder of the appendix. Let also for the remainder of the text $\Gamma \subset G$ be a fixed arbitrary finite subgroup. Let $\mathcal{H}$ denote the set of all reductive subgroups $H$ of $G$ with same rank as $G$. By \Cref{equivalence.of.fans.and.abelian.groups} the induction assumption also implies that we assume the existence of $N(H)$ for any reductive group $H$ with $\dim H< \dim G$. We define $N$ as 
\begin{equation*}
\begin{aligned}
    N&\coloneqq \max_{H\in \mathcal{H}} \{ N(H) \}. 
\end{aligned}
\end{equation*} 
We see that $N$ is finite since there are a finite number of conjugacy classes of reductive equal rank subgroups. To clarify this, there are a finite number of conjugacy classes of connected equal rank subgroups, since these correspond to closed subsystems of the root system of $G$. Furthermore, for any such connected $H_0$, there are a finite number of groups $H$ such that $H^{\circ}=H_0$, since $H$ must normalize $Z(H_0)$, and $[N_G(Z(H_0)):C_G(Z(H_0))]$ is finite. We also fix an integer $M>N$, which we may assume to be equal to $N+1$.
\begin{lemma}\label{m.fan.lemma1}
    An $M$-fan $F$ in $\Gamma$ not contained in $Z(G)$ is contained in a unique maximal $M$-fan in $\Gamma$.
\end{lemma}
\begin{proof}
Let $F$ be an $M$-fan in $\Gamma$ not contained in $\rZ(G)$. Pick any $f\in F$ not contained in $\rZ(G)$. 
Since $F$ is abelian, we have that $F \subset C_{\Gamma}(f)$. Thus, $F$ is an $M$-fan in $C_{\Gamma}(f)$. Since $f\notin \rZ(G)$ we have that $\dim C_G(f)< \dim G$. Thus, by the ongoing induction assumption on $\dim G$ we see that $F$ is contained in a unique maximal $M$-fan $\calf$ in $C_{\Gamma}(f)$.
Since $M$-fans are abelian, any $M$-fan containing $F$ also centralizes $f$, which means it is contained in $\mathcal{F}$.
\end{proof}

We say that an $M$-fan $F$ is \textit{associated} to a noncentral element $g \in \Gamma$ if it is the unique maximal $M$-fan $C_{\Gamma}(g)$. We write $F_g$ for the $M$-fan associated to $g$. In the sequel, we let
\begin{equation*}
    Z\coloneqq Z(G) \cap \Gamma.
\end{equation*}
This is an $M$-fan.

In the notation of the proof of \Cref{m.fan.lemma1}, since $F\subset \calf$ and both are abelian, we see that $\mathcal{F} \subset C_{\Gamma}(F)$. The group $C_{\Gamma}(F)$ is reductive and of the same rank as $G$, which means by induction that $\mathcal{F}$ is maximal in $C_{\Gamma}(F)$ and thus
\begin{equation}\label{calf.n.inequality}
    [C_{\Gamma}(F):\mathcal{F}] \leq N(C_{\Gamma}(F))\leq  N
\end{equation}
We can now prove some simple statements.
\begin{lemma}\label{non.maximal.mfan}
Let $F$ be an $M$-fan such that $Z(G) \subsetneq F \subsetneq \mathcal{F}$, where $\calf$ is a maximal $M$-fan in $\Gamma$. The number $n_F$ of elements $g$ of $\Gamma$ associated with $F$ is divisible by $|\mathcal{F}|$ and
    \[\frac{n_F}{|\mathcal{F}|} \leq N.\]
\end{lemma}
\begin{proof}
We can assume $n_F \geq 1$, since the statement is trivial otherwise. All elements associated with $F$ lie in $C_{\Gamma}(F)$. Furthermore, if $g$ is associated to $F$ and if $f \in \mathcal{F}$, then the element $gf$ is still associated with $F$. To see this, note that we cannot have $gf \in Z(G)$, since this would mean $C_{\Gamma}(f)=C_{\Gamma}(g)$. Since $\calf \subset C_{\Gamma}(f)$ and since $F$ is the maximal $M$-fan of $C_{\Gamma}(g)$, we see that this cannot happen. Consider $F_{gf}$, the maximal $M$-fan of $C_G(gf)$. We have that $F \subset C_{\Gamma}(gf)$, which means $F \subset F_{gf}$. Since $\mathcal{F}$ is the unique maximal $M$-fan containing $F$, we must have that $F_{gf} \subset \mathcal{F}$. This means that $F_{gf}$ must commute with $f$ and therefore also $g$. We thus see that $F_{gf} \subset C_{\Gamma}(g)$, whence $F_{gf}=F$. That is, $gf$ is associated with $F$.

The set of elements associated to $F$ are thus unions of cosets of $\mathcal{F}$, and they all lie in $C_\Gamma(F)$. By \Cref{calf.n.inequality}, the index of $\mathcal{F}$ in $C_\Gamma(F)$ is at most $N$, from which the statement of the lemma follows.
\end{proof}
We get also get a counting formula for maximal $M$-fans in $\Gamma$.
\begin{lemma}\label{maximal.fan}
 Let $\mathcal{F}$ be a maximal $M$-fan in $\Gamma$. The number of elements of $\Gamma\setminus Z$ associated to $\mathcal{F}$ is 
    \[n_{\mathcal{F}}=|C_{\Gamma}(\mathcal{F})|-|Z|.\]
    and $[C_{\Gamma}(\mathcal{F}):\mathcal{F}] \leq N$.
\end{lemma}
\begin{proof}
The inequality on the index follows by the same argument as for \Cref{calf.n.inequality}. Too see that $n_{\mathcal{F}}=|C_{\Gamma}(\mathcal{F})|-|Z|$, note that 
an element $g \in \Gamma$ is associated to $\mathcal{F}$ if and only if $\mathcal{F} \subset C_{\Gamma}(g)$, which is equivalent to  $g \in C_{\Gamma}(\mathcal{F})$. A maximal $M$-fan will contain $Z$, since if $F$ is an $M$-fan in $\Gamma$, $ZF$ is an $M$-fan as well. Since $Z\subset C_{\Gamma}(\calf)$, the statement of the lemma follows. 
\end{proof}

We now begin to produce a formula for the number of elements in $\Gamma$, which will allow us to finalize the proof of \Cref{general.jordan}. To do this, we split the elements in $\Gamma$ into four sets,
\[\Gamma=Z \cup \Gamma_1 \cup \Gamma_2 \cup \Gamma_3.\]
Here $\Gamma_1$ consists of the noncentral elements associated to $Z$, $\Gamma_2$ consists of the elements associated to an $M$-fan $F$ such that $Z\subsetneq F \subsetneq \mathcal{F}$, where $\mathcal{F}$ is a maximal $M$-fan containing $F$, and $\Gamma_3$ consists of the elements associated to a maximal $M$-fan $\mathcal{F}$.

We begin by enumerating $\Gamma_1$. It is a union of cosets of $Z$. Indeed, if $g \in \Gamma$ is associated to $Z$, then so is $zg$ for $z \in Z$, since $C_{\Gamma}(g)=C_{\Gamma}(zg)$. Furthermore, $\Gamma$ permutes these cosets by conjugation, since $\gamma gZ \gamma^{-1}=\gamma g \gamma^{-1}Z$, for all $g, \gamma \in \Gamma$.

Let $S=G/[G,G]$ and let $\rho:G \to S$ be the quotient map. Let 
\begin{equation*}
    n\coloneqq |[G:G] \cap Z(G)|,
\end{equation*}
a finite number. We claim that for any $g \in \Gamma$, 
\begin{equation}\label{n.c.index}
    [N_{\Gamma}(gZ) : C_{\Gamma}(g)] \leq n.
\end{equation}
To see this, note that for any $h \in N_{\Gamma}(gZ)$ there is a unique $z\in Z$ such that $hgh^{-1}=gz$. This implies that $\rho(z)=1$, so $z \in [G:G] \cap Z$. This explains the inequality in (\ref{n.c.index}). Counting all elements in the $\Gamma$-conjugacy class of the coset $gZ$, we see that
\[|\text{$\Gamma$-orbit of $gZ$}||Z|=\frac{|\Gamma||Z|}{N_{\Gamma}(gZ)}=|\Gamma|\frac{|C_{\Gamma}(g)|}{|N_{\Gamma}(gZ)|}\frac{|Z|}{|C_{\Gamma}(g)|}\]
\[=|\Gamma|\frac{1}{[N_{\Gamma}(gZ):C_{\Gamma}(g)][C_{\Gamma}(g):Z]}=|\Gamma|\frac{1}{\lambda},\]
where $\lambda \coloneqq \frac{1}{[N_{\Gamma}(gZ):C_{\Gamma}(g)][C_{\Gamma}(g):Z]}$. 
By combining all such conjugacy classes we see that
\[|\Gamma_1|=|\Gamma|\left(\frac{1}{\lambda_1}+\frac{1}{\lambda_2}+...+\frac{1}{\lambda_{k_1}}\right)\]
where, for each $i=1,...,k_1$, we have that $\lambda_i \leq nN$. 

We continue by counting the elements of $\Gamma_2$. If $F$ is an $M$-fan in $\Gamma$ strictly contained in a maximal $M$-fan $\mathcal{F}$, we know that if $n_F$ is the number of elements associated to $F$, then by \Cref{non.maximal.mfan} $\frac{n_F}{|\mathcal{F}|}$ is an integer less than or equal to $N$. The number of $M$-fans conjugate to $F$ is $\frac{|\Gamma|}{|N_{\Gamma}(F)|}$. We have that
\[[N_{\Gamma}(F):C_{\Gamma}(F)] \leq (\dim G)!\]
 since $N_{\Gamma}(F)$ permute the eigenspaces of $F$ and $C_{\Gamma}(F)$ are the elements in $\Gamma$ preserving the eigenspace decomposition of $F$, and $\dim \mathfrak{g}=\dim G$. Combining what we have above we get that the number of $g \in \Gamma$ associated to an $M$-fan in the conjugacy class of $F$ is
\[ |\text{$\Gamma$-orbit of $F$}|n_{F}=\frac{|\Gamma|n_{F}}{|N_{\Gamma}(F)|}=|\Gamma|\frac{|C_{\Gamma}(F)|}{|N_{\Gamma}(F)|}\frac{n_F}{|\mathcal{F}|[C_{\Gamma}(F):\mathcal{F}]}=|\Gamma|\frac{\nu}{\mu}\]
where $\nu=\frac{n_F}{|\mathcal{F}|}$ and $\mu=[N_{\Gamma}(F):C_{\Gamma}(F)][C_{\Gamma}(F):\mathcal{F}]$. Summming over all conjugacy classes of $M$-fans we see that there are integers $\nu_i, \mu_i$ with $\nu_i \leq N$ and $\mu_i \leq N(\dim G)!$ such that
\[|\Gamma_2|=|\Gamma|\left( \frac{\nu_1}{\mu_1}+\frac{\nu_2}{\mu_2}+...+\frac{\nu_{k_2}}{\mu_{k_2}} \right).\]

Lastly, we count the elements in $\Gamma_3$. Let $\mathcal{F}$ be a maximal $M$-fan. By \Cref{maximal.fan} we know the number of elements $n_{\mathcal{F}}$ associated to $\mathcal{F}$ is $|C_{\Gamma}(\mathcal{F})|-|Z|$. If we let $\omega=[N_{\Gamma}(\mathcal{F}):C_{\Gamma}(\mathcal{F})]$ and $q=[N_{\Gamma}(\mathcal{F}):Z]$, we see that the number of elements associated to the conjugacy class of $\mathcal{F}$ is
\[|\text{$\Gamma$-orbit of $\mathcal{F}$}|(|C_{\Gamma}(\mathcal{F})|-|Z|)=|\Gamma|\left( \frac{|C_{\Gamma}(\mathcal{F})|}{|N_{\Gamma}(\mathcal{F})|}-\frac{|Z|}{|N_{\Gamma}(\mathcal{F})|}\right)=|\Gamma|\left( \frac{1}{w}-\frac{1}{q}\right).\]
Summing over all conjugacy classes of all maximal $M$-fans we see that there are integers $\omega_i, q_i$, with $q_i$ dividing $\frac{|\Gamma|}{|Z|}$, such that $\omega_i \leq (\dim G)!$, $q_i = \omega_i[C_{\Gamma}(\calf):Z] \geq 2\omega_i$ and
\[|\Gamma_3|=|\Gamma|\left(\frac{1}{\omega_1}-\frac{1}{q_1}+\frac{1}{\omega_2}-\frac{1}{q_2}+...+\frac{1}{\omega_{k_3}}-\frac{1}{q_{k_3}} \right).\]
We now combine our three enumerations to get the following formula for the order of $\Gamma$
\begin{equation}\label{fundamental.equation}
|\Gamma|=|Z|+|\Gamma|\left(\sum_{i=1}^{k_1}\frac{1}{\lambda_i}+\sum_{i=1}^{k_2}\frac{\nu_i}{\mu_i}+\sum_{i=1}^{k_3}\left(\frac{1}{\omega_i}-\frac{1}{q_i}\right) \right).
\end{equation}
Here, $\lambda_i,\nu_i,\mu_i,\omega_i$ are all positive integers bounded above by $C=2nN(\dim G)!$ and $q_i$ divides $\frac{|\Gamma|}{|Z|}$. 
Furthermore, since $1/\omega_i-1/q_i \geq 1/2\omega_i$ we see that $k_1+k_2+k_3 \leq C$. In particular, the $k_i$ are also bounded above by $C$. We gather terms and obtain the following statement.
\begin{lemma}\label{intermediate.equation}
    If $\gamma\coloneqq |\Gamma|/|Z|$, then there are positive integers $q_1, ..., q_k$ for some $k_3\geq 1$ such that
    \[\frac{1}{\gamma}=\left(\sum_{i=1}^{k_3}\frac{1}{q_i}\right) -\frac{b}{a},\]
    where $b/a$ is a reduced fraction and $b,a$ are integers both bounded by a constant depending only on $G$.
\end{lemma}
\begin{proof}
This is the conclusion of the counting arguments preceding the lemma.
\end{proof}

\begin{lemma}\label{little.equation}
Suppose that we have positive integers $a,b,\gamma, q_i$ with $q_i\leq \gamma$ an equation of the form
    \[\frac{1}{\gamma}=\left(\sum_{i=1}^{k}\frac{1}{q_i}\right)-\frac{b}{a}.\]
    Then $\gamma \leq f(k,a)$, where $f:\mathbb{N}^2 \to \mathbb{N}$ is some function. 
\end{lemma}

\begin{proof}   
This is \cite[Lemma 4.1]{breuillard.jordan.exposition}.
\end{proof}

\begin{proof}[Proof of \Cref{general.jordan}]
 Let 
\begin{equation*}
    f_G \coloneqq \max\{ a,N(\dim G)!,f(k_3,a) \},
\end{equation*}
where $f\colon \mathbb{N}^2\to \mathbb{N}$ is the function obtained from \Cref{intermediate.equation} and \Cref{little.equation}. 

    Consider $\Cref{fundamental.equation}$. 
    
    If $k_3=0$, then in $\Cref{intermediate.equation}$ we obtain that $1/\gamma=-b/a$, which implies that $b=-1$ and thus $\gamma=\frac{|\Gamma|}{|Z|}=a$ which is bounded only in terms of $G$. This means that $Z$ is of bounded index in $\Gamma$, and we can thus choose $A=Z$. 
    
    If instead $k_3 > 0$ and $q_i=\gamma$ for some $i$, then for some maximal $M$-beam $\mathcal{F}_i$ we have that $[N_{\Gamma}(\mathcal{F}_i):Z]=q_i=\gamma=[\Gamma:Z]$, which means that $\Gamma=N_{\Gamma}(\mathcal{F}_i)$. We thus choose $A=\mathcal{F}_i$, which is uniformly bounded since 
    \[[\Gamma:\mathcal{F}_i] = [N_{\Gamma}(\mathcal{F}_i):C_{\Gamma}(\mathcal{F}_i)][C_{\Gamma}(\mathcal{F}_i):\mathcal{F}_i] \leq N(\dim G)!.\]
    
    Now assume that $q_i \neq \gamma$ for all $i$ and consider the non-negative integer \begin{equation*}
        \gamma \left(\sum_{i=1}^{k_3}\frac{1}{q_i}\right)-1.
    \end{equation*} 
    If this is zero, then $\sum_{i=1}^{k_3}\frac{1}{q_i}=\frac{1}{\gamma}$ which, since $q_i\leq \gamma$ implies that $k_3=1$ and $q_1=\gamma$, a contradiction. If it instead is positive, then $b>0$ and we can apply \Cref{little.equation}, in which case $\gamma$ is again bounded only in terms of $G$, this time by $f(k_3,a)$, and we again choose $A=Z$.
\end{proof}
%
%
%
%
%
%
%
%
%
%
%
\printbibliography

@Article{sikora,
      title={Character Varieties}, 
      author={Adam S. Sikora},
      year={2010},
      eprint={0902.2589},
      archivePrefix={arXiv},
      primaryClass={math.RT}
}

@Article{raman1,
 Author = {Ramanathan, A.},
 Title = {Moduli for principal bundles over algebraic curves. {I}},
 FJournal = {Proceedings of the Indian Academy of Sciences. Mathematical Sciences},
 Journal = {Proc. Indian Acad. Sci., Math. Sci.},
 ISSN = {0253-4142},
 Volume = {106},
 Number = {3},
 Pages = {301--328},
 Year = {1996},
 Language = {},
 DOI = {10.1007/BF02867438},
 zbMATH = {955213},
 Zbl = {0901.14007}
}

@Article{raman2,
 Author = {Ramanathan, A.},
 Title = {Moduli for principal bundles over algebraic curves. {II}},
 FJournal = {Proceedings of the Indian Academy of Sciences. Mathematical Sciences},
 Journal = {Proc. Indian Acad. Sci., Math. Sci.},
 ISSN = {0253-4142},
 Volume = {106},
 Number = {4},
 Pages = {421--449},
 Year = {1996},
 Language = {},
 DOI = {10.1007/BF02837697},
 zbMATH = {1043990},
 Zbl = {0901.14008}
}

@Article{balaji.semi.princ,
 Author = {Balaji, V. and Seshadri, C. S.},
 Title = {Semistable principal bundles. {I}: {Characteristic} zero.},
 FJournal = {Journal of Algebra},
 Journal = {J. Algebra},
 ISSN = {0021-8693},
 Volume = {258},
 Number = {1},
 Pages = {321--347},
 Year = {2002},
 Language = {},
 DOI = {10.1016/S0021-8693(02)00502-1},
 zbMATH = {1868091},
 Zbl = {1099.14502}
}

@Article{balaji2,
 Author = {Balaji, V. and Parameswaran, A. J.},
 Title = {Semistable principal bundles. {II}: {Positive} characteristics},
 FJournal = {Transformation Groups},
 Journal = {Transform. Groups},
 ISSN = {1083-4362},
 Volume = {8},
 Number = {1},
 Pages = {3--36},
 Year = {2003},
 Language = {},
 DOI = {10.1007/s00031-003-0713-2},
 zbMATH = {1925607},
 Zbl = {1084.14013}
}

@Article{balaji.seshadri.bruhat.tits,
 Author = {Balaji, V. and Seshadri, C. S.},
 Title = {Moduli of parahoric {{\(\mathcal G\)}}-torsors on a compact {Riemann} surface},
 FJournal = {Journal of Algebraic Geometry},
 Journal = {J. Algebr. Geom.},
 ISSN = {1056-3911},
 Volume = {24},
 Number = {1},
 Pages = {1--49},
 Year = {2015},
 Language = {},
 DOI = {10.1090/S1056-3911-2014-00626-3},
 zbMATH = {6387932},
 Zbl = {1330.14059}
}

@Article{heinloth.semistable.reduction,
 Author = {Heinloth, Jochen},
 Title = {Semistable reduction for {{\(G\)}}-bundles on curves},
 FJournal = {Journal of Algebraic Geometry},
 Journal = {J. Algebr. Geom.},
 ISSN = {1056-3911},
 Volume = {17},
 Number = {1},
 Pages = {167--183},
 Year = {2008},
 Language = {},
 DOI = {10.1090/S1056-3911-07-00476-6},
 zbMATH = {5224493},
 Zbl = {1186.14035}
}

@misc{AHLH,
  doi = {10.48550/ARXIV.1812.01128},  
  url = {https://arxiv.org/abs/1812.01128},  
  author = {Alper, Jarod and Halpern-Leistner, Daniel and Heinloth, Jochen},  
  title = {Existence of moduli spaces for algebraic stacks},  
  publisher = {arXiv},  
  year = {2018},  
  copyright = {Creative Commons Attribution Non Commercial No Derivatives 4.0 International}
}

@article {hyeon-murphy,
    AUTHOR = {Hyeon, Donghoon and Murphy, David},
     TITLE = {Note on the stability of principal bundles},
   JOURNAL = {Proc. Amer. Math. Soc.},
  FJOURNAL = {Proceedings of the American Mathematical Society},
    VOLUME = {132},
      YEAR = {2004},
    NUMBER = {8},
     PAGES = {2205--2213},
      ISSN = {0002-9939},
   MRCLASS = {14H60 (14D20)},
  MRNUMBER = {2052395},
MRREVIEWER = {A. \'{A}lvarez V\'{a}zquez},
       DOI = {10.1090/S0002-9939-04-07386-1},
}

@misc{stacks-project,
  author       = {The {Stacks project authors}},
  title        = {The Stacks project},
  howpublished = {\url{https://stacks.math.columbia.edu}},
  year         = {2023},
}

@article {alper-gms,
    AUTHOR = {Alper, Jarod},
     TITLE = {Good moduli spaces for {A}rtin stacks},
   JOURNAL = {Ann. Inst. Fourier (Grenoble)},
  FJOURNAL = {Universit\'{e} de Grenoble. Annales de l'Institut Fourier},
    VOLUME = {63},
      YEAR = {2013},
    NUMBER = {6},
     PAGES = {2349--2402},
      ISSN = {0373-0956},
   MRCLASS = {14D23 (14L24 14L30)},
  MRNUMBER = {3237451},
MRREVIEWER = {Arvid Siqveland},
       URL = {http://aif.cedram.org/item?id=AIF_2013__63_6_2349_0},
}

@article {ghiasabadi.reppen,
    AUTHOR = {Ghiasabadi, Archia and Reppen, Stefan},
     TITLE = {Essentially finite {$G$}-torsors},
   JOURNAL = {Bull. Sci. Math.},
  FJOURNAL = {Bulletin des Sciences Math\'{e}matiques},
    VOLUME = {188},
      YEAR = {2023},
     PAGES = {Paper No. 103334},
      ISSN = {0007-4497},
   MRCLASS = {14H60 (14A20 14L30 20G05)},
  MRNUMBER = {4644771},
       DOI = {10.1016/j.bulsci.2023.103334},
       URL = {},
}

@article {heinloth.hilbert-mumford,
    AUTHOR = {Heinloth, Jochen},
     TITLE = {Hilbert-{M}umford stability on algebraic stacks and
              applications to {$\mathcal{G}$}-bundles on curves},
   JOURNAL = {\'{E}pijournal G\'{e}om. Alg\'{e}brique},
  FJOURNAL = {\'{E}pijournal de G\'{e}om\'{e}trie Alg\'{e}brique. EPIGA},
    VOLUME = {1},
      YEAR = {2017},
     PAGES = {Art. 11, 37},
   MRCLASS = {14L24 (14A20 14D23)},
  MRNUMBER = {3758902},
MRREVIEWER = {Scott R. Nollet},
       DOI = {10.46298/epiga.2018.volume1.2062},
}

@article {heinloth.uniformization,
    AUTHOR = {Heinloth, Jochen},
     TITLE = {Uniformization of {$\mathscr{G}$}-bundles},
   JOURNAL = {Math. Ann.},
  FJOURNAL = {Mathematische Annalen},
    VOLUME = {347},
      YEAR = {2010},
    NUMBER = {3},
     PAGES = {499--528},
      ISSN = {0025-5831},
   MRCLASS = {14D23 (14D20)},
  MRNUMBER = {2640041},
MRREVIEWER = {Ilya Karzhemanov},
       DOI = {10.1007/s00208-009-0443-4},
}

@article {behrend,
    AUTHOR = {Behrend, Kai A.},
     TITLE = {Semi-stability of reductive group schemes over curves},
   JOURNAL = {Math. Ann.},
  FJOURNAL = {Mathematische Annalen},
    VOLUME = {301},
      YEAR = {1995},
    NUMBER = {2},
     PAGES = {281--305},
      ISSN = {0025-5831},
   MRCLASS = {20G35},
  MRNUMBER = {1314588},
MRREVIEWER = {Zhi Jie Chen},
       DOI = {10.1007/BF01446630},
}

@misc{weissmann-zhang,
      title={A stacky approach to identify the semi-stable locus of vector bundles}, 
      author={Dario Weissmann and Xucheng Zhang},
      year={2023},
      eprint={2302.09245},
      archivePrefix={arXiv},
      primaryClass={math.AG}
}

@misc{weissmann,
      title={A functorial approach to the stability of vector bundles}, 
      author={Dario Weissmann},
      year={2022},
      eprint={2211.08260},
      archivePrefix={arXiv},
      primaryClass={math.AG}
}

@Article{pappas.rapoport,
 Author = {Pappas, G. and Rapoport, M.},
 Title = {Twisted loop groups and their affine flag varieties. {With} an appendix by {T}. {Haines} and {M}. {Rapoport}.},
 FJournal = {Advances in Mathematics},
 Journal = {Adv. Math.},
 ISSN = {0001-8708},
 Volume = {219},
 Number = {1},
 Pages = {118--198},
 Year = {2008},
 Language = {},
 DOI = {10.1016/j.aim.2008.04.006},
 zbMATH = {5313909},
 Zbl = {1159.22010}
}

@Article{deninger,
 Author = {Deninger, Ch.},
 Title = {A proper base change theorem for non-torsion sheaves in {\'e}tale cohomology},
 FJournal = {Journal of Pure and Applied Algebra},
 Journal = {J. Pure Appl. Algebra},
 ISSN = {0022-4049},
 Volume = {50},
 Number = {3},
 Pages = {231--235},
 Year = {1988},
 Language = {},
 DOI = {10.1016/0022-4049(88)90102-8},
 zbMATH = {4099464},
 Zbl = {0672.14010}
}

@Article{faltings,
 Author = {Faltings, Gerd},
 Title = {Stable {{\(G\)}}-bundles and projective connections},
 FJournal = {Journal of Algebraic Geometry},
 Journal = {J. Algebr. Geom.},
 ISSN = {1056-3911},
 Volume = {2},
 Number = {3},
 Pages = {507--568},
 Year = {1993},
 Language = {},
 zbMATH = {404371},
 Zbl = {0790.14019}
}

@Article{hein,
 Author = {Hein, Georg},
 Title = {Generalized {Albanese} morphisms},
 FJournal = {Compositio Mathematica},
 Journal = {Compos. Math.},
 ISSN = {0010-437X},
 Volume = {142},
 Number = {3},
 Pages = {719--733},
 Year = {2006},
 Language = {},
 DOI = {10.1112/S0010437X05001831},
 zbMATH = {5042492},
 Zbl = {1100.14028}
}

@Article{jordan,
 Author = {Jordan, C.},
 Title = {M{\'e}moire sur les {\'e}quations diff{\'e}rentielles lin{\'e}aires {\`a} int{\'e}grale alg{\'e}brique.},
 FJournal = {Journal f{\"u}r die Reine und Angewandte Mathematik},
 Journal = {J. Reine Angew. Math.},
 ISSN = {0075-4102},
 Volume = {84},
 Pages = {89--215},
 Year = {1877},
 Language = {French},
 zbMATH = {2712557},
 JFM = {09.0234.01}
}

@Article{ducrohet.mehta,
 Author = {Ducrohet, Laurent and Mehta, Vikram B.},
 Title = {Density of vector bundles periodic under the action of {Frobenius}},
 FJournal = {Bulletin des Sciences Math{\'e}matiques},
 Journal = {Bull. Sci. Math.},
 ISSN = {0007-4497},
 Volume = {134},
 Number = {5},
 Pages = {454--460},
 Year = {2010},
 Language = {},
 DOI = {10.1016/j.bulsci.2009.11.001},
 zbMATH = {5768877},
 Zbl = {1195.14022}
}

@InCollection{gomez.langer.schmitt.sols,
 Author = {G{\'o}mez, Tom{\'a}s L. and Langer, Adrian and Schmitt, Alexander H. W. and Sols, Ignacio},
 Title = {Moduli spaces for principal bundles in large characteristics},
 BookTitle = {Teichm\"uller theory and moduli problem. Proceedings of the workshop at Harish-Chandra Research Institute, Allahabad, India, January 5--15, 2006},
 ISBN = {978-93-80416-00-7},
 Pages = {281--371},
 Year = {2010},
 Publisher = {Mysore: Ramanujan Mathematical Society},
 Language = {},
 zbMATH = {5778464},
 Zbl = {1201.14008}
}

@misc{behrend.thesis,
      title={The Lefschetz trace formula for algebraic stacks}, 
      author={Kai Behrend},
      year={1990},
    publisher={PhD Thesis at University of Californa, Berkeley},
      eprint={https://personal.math.ubc.ca/~behrend/thesis.pdf},
      archivePrefix={},
      primaryClass={math.AG}
}

@Book{dolgachev,
 Author = {Dolgachev, Igor},
 Title = {Lectures on invariant theory},
 FSeries = {London Mathematical Society Lecture Note Series},
 Series = {Lond. Math. Soc. Lect. Note Ser.},
 ISSN = {0076-0552},
 Volume = {296},
 ISBN = {0-521-52548-9},
 Year = {2003},
 Publisher = {Cambridge: Cambridge University Press},
 Language = {},
 zbMATH = {1822312},
 Zbl = {1023.13006}
}

@InCollection{conrad.reductive.group.schemes,
 Author = {Conrad, Brian},
 Title = {Reductive group schemes},
 BookTitle = {Autour des sch\'emas en groupes. \'Ecole d'\'Et\'e ``Sch\'emas en groupes''. Volume I},
 ISBN = {978-2-85629-794-0},
 Pages = {93--444},
 Year = {2014},
 Publisher = {Paris: Soci{\'e}t{\'e} Math{\'e}matique de France (SMF)},
 Language = {},
 zbMATH = {6479627},
 Zbl = {1349.14151}
}

@misc{alper.stacksandmoduli,
  author       = {Alper, Jarod},
  title        = {Stacks and Moduli},
  howpublished = {\url{https://sites.math.washington.edu/~jarod/moduli.pdf}},
  year         = {2023},
}

@Article{Nori1,
 Author = {Nori, Madhav V.},
 Title = {On the representations of the fundamental group},
 FJournal = {Compositio Mathematica},
 Journal = {Compos. Math.},
 ISSN = {0010-437X},
 Volume = {33},
 Pages = {29--41},
 Year = {1976},
 Language = {},
 zbMATH = {3526839},
 Zbl = {0337.14016}
}

@Article{Nori2,
 Author = {Nori, Madhav V.},
 Title = {The fundamental group-scheme},
 FJournal = {Proceedings of the Indian Academy of Sciences. Mathematical Sciences},
 Journal = {Proc. Indian Acad. Sci., Math. Sci.},
 ISSN = {0253-4142},
 Volume = {91},
 Pages = {73--122},
 Year = {1982},
 Language = {},
 DOI = {10.1007/BF02967978},
 zbMATH = {3939489},
 Zbl = {0586.14006}
}

@misc{breuillard.jordan.exposition,
      title={An exposition of {Jordan}'s original proof of his theorem on finite subgroups of $\GL_n(\mathbb{C})$}, 
      author={Emmanuel Breuillard},
      year={2023},
      eprint={https://arxiv.org/abs/2305.00722},
      archivePrefix={arXiv},
      primaryClass={math.AG}
}

@InCollection{esnault.et.al,
 Author = {Esnault, H. and Hai, P. H. and Sun, X.},
 Title = {On {Nori}'s fundamental group scheme},
 BookTitle = {Geometry and dynamics of groups and spaces. In memory of Alexander Reznikov. 2006},
 ISBN = {},
 Pages = {377--398},
 Year = {2008},
 Publisher = {Basel: Birkh{\"a}user},
 Language = {},
 zbMATH = {5269424},
 Zbl = {1137.14035}
}

@Article{borne.vistoli,
 Author = {Borne, N. and Vistoli, A.},
 Title = {The {Nori} fundamental gerbe of a fibered category},
 FJournal = {Journal of Algebraic Geometry},
 Journal = {J. Algebr. Geom.},
 ISSN = {},
 Volume = {24},
 Number = {2},
 Pages = {311--353},
 Year = {2015},
 Language = {},
 %DOI = {10.1090/S1056-3911-2014-00638-X},
 Keywords = {14A20,14F05,18D05},
 zbMATH = {6434531},
 Zbl = {1349.14004}
}

@Article{biswas.dossantos,
 Author = {Biswas, I. and dos Santos, J. P. P.},
 Title = {Vector bundles trivialized by proper morphisms and the fundamental group scheme},
 FJournal = {Journal of the Institute of Mathematics of Jussieu},
 Journal = {J. Inst. Math. Jussieu},
 ISSN = {},
 Volume = {10},
 Number = {2},
 Pages = {225--234},
 Year = {2011},
 Language = {},
 %DOI = {10.1017/S1474748010000071},
 Keywords = {14L15,14F05},
 zbMATH = {5887644},
 Zbl = {1214.14037}
}

@article {antei.emsalem.gasbarri,
    AUTHOR = {Antei, Marco and Emsalem, Michel and Gasbarri, Carlo},
     TITLE = {Sur l'existence du sch\'{e}ma en groupes fondamental},
   JOURNAL = {\'{E}pijournal G\'{e}om. Alg\'{e}brique},
  FJOURNAL = {\'{E}pijournal de G\'{e}om\'{e}trie Alg\'{e}brique. EPIGA},
    VOLUME = {4},
      YEAR = {2020},
     PAGES = {Art. 5, 15},
      ISSN = {2491-6765},
   MRCLASS = {14F35 (11G25 14F20 14L15 14L30)},
  MRNUMBER = {4113656},
MRREVIEWER = {Christopher\ David\ Lazda},
       DOI = {10.46298/epiga.2020.volume4.5436},
}

@Article{drinfeld.simpson,
 Author = {Drinfeld, V. G. and Simpson, Carlos},
 Title = {{{\(B\)}}-structures on {{\(G\)}}-bundles and local triviality},
 FJournal = {Mathematical Research Letters},
 Journal = {Math. Res. Lett.},
 ISSN = {1073-2780},
 Volume = {2},
 Number = {6},
 Pages = {823--829},
 Year = {1995},
 Language = {},
 DOI = {10.4310/MRL.1995.v2.n6.a13},
 zbMATH = {858528},
 Zbl = {0874.14043}
}
\end{document}